\documentclass[11pt]{article}

\usepackage{IEEEtrantools}
\usepackage{latexsym}
\usepackage{amsfonts}
\usepackage{amssymb}
\usepackage{amsmath,amsfonts,amsthm}
\usepackage{mathrsfs}
\usepackage[all]{xy}
\usepackage{epic}
\usepackage{eepic}
\usepackage{enumerate}
\usepackage{multirow}
\usepackage[breaklinks]{hyperref}
\usepackage{tikz}

\usepackage{mathrsfs}
\usepackage[all]{xy}
\usepackage{epic}
\usepackage{eepic}
\usepackage{enumerate}
\usepackage{multirow}
\usepackage[breaklinks]{hyperref}
\usepackage{tikz}
\usetikzlibrary{matrix,arrows,decorations.pathmorphing}
\usepackage{color}

\newcommand{\bbbt}{\mathbb{T}}
\newcommand{\scrt}{\mathscr{T}}

\newcommand{\be}{\begin{equation}}
	\newcommand{\ee}{\end{equation}}
\newcommand{\bea}{\begin{eqnarray}}
	\newcommand{\eea}{\end{eqnarray}}
\newcommand{\bean}{\begin{eqnarray*}}
	\newcommand{\eean}{\end{eqnarray*}}
\newcommand{\brray}{\begin{array}}
	\newcommand{\erray}{\end{array}}
\newcommand{\biearray}{\begin{IEEEarray}{rCl}}
	\newcommand{\eiearray}{\end{IEEEarray}}

\newcommand{\newsection}[1]{\setcounter{equation}{0}
	\setcounter{dfn}{0}
	\section{#1}}

\newtheorem{dfn}{Definition}[section]
\newtheorem{thm}[dfn]{Theorem}
\newtheorem{lmma}[dfn]{Lemma}
\newtheorem{ppsn}[dfn]{Proposition}
\newtheorem{crlre}[dfn]{Corollary}
\newtheorem{xmpl}[dfn]{Example}
\newtheorem{rmrk}[dfn]{Remark}

\newcommand{\bdfn}{\begin{dfn}\rm}
	\newcommand{\bthm}{\begin{thm}}
		\newcommand{\blmma}{\begin{lmma}}
			\newcommand{\bppsn}{\begin{ppsn}}
				\newcommand{\bcrlre}{\begin{crlre}}
					\newcommand{\bxmpl}{\begin{xmpl}}
						\newcommand{\brmrk}{\begin{rmrk}\rm}
							
							\newcommand{\edfn}{\end{dfn}}
						\newcommand{\ethm}{\end{thm}}
					\newcommand{\elmma}{\end{lmma}}
				\newcommand{\eppsn}{\end{ppsn}}
			\newcommand{\ecrlre}{\end{crlre}}
		\newcommand{\exmpl}{\end{xmpl}}
	\newcommand{\ermrk}{\end{rmrk}}

\newcommand{\bbc}{\mathbb{C}}

\newcommand{\bbn}{\mathbb{N}}

\newcommand{\prf}{\noindent{\it Proof\/}: }

\def \qed { \mbox{}\hfill
	$\Box$\vspace{1ex}}

\newcommand\restr[2]{{
		\left.\kern-\nulldelimiterspace 
		\littletaller
		\right|_{#2} 
}}

\newcommand{\littletaller}{\mathchoice{\vphantom{\big|}}{}{}{}}

\begin{document}
	
	\author{{\sc Akshay Bhuva, Bipul Saurabh}}
	\title{Computation of Gelfand-Kirillov dimension  for $B$-type structures}
\maketitle

	\begin{abstract}  
		Let $\mathcal{O}(\mbox{Spin}_{q^{1/2}}(2n+1))$ and $\mathcal{O}(SO_q(2n+1))$ be the quantized algebras of regular functions on the Lie groups   $\mbox{Spin}(2n+1)$ and  $SO(2n+1)$, respectively.  In this article, we prove that the Gelfand-Kirillov dimension of a  simple unitarizable $\mathcal{O}(\mbox{Spin}_{q^{1/2}}(2n+1))$-module $V_{t,w}^{\mbox{Spin}}$ is the same as the length of the  Weyl word $w$.  We show that the same result holds for the $\mathcal{O}(SO_q(2n+1))$-module $V_{t,w}$, which is  obtained from $V_{t,w}^{\mbox{Spin}}$ by  restricting the algebra action  to the subalgebra $\mathcal{O}(SO_q(2n+1))$ of $\mathcal{O}(\mbox{Spin}_{q^{1/2}}(2n+1))$. 
		Moreover, we  consider  the quantized algebras of regular functions  on  certain homogeneous spaces of $SO(2n+1)$ and $\mbox{Spin}(2n+1)$  and show that its Gelfand-Kirillov dimension is equal to  the dimension of the homogeneous space as a real differentiable manifold.

	\end{abstract}


	\section{Introduction}

	The Gelfand-Kirillov dimension (abbreviated as GKdim) is  a useful tool for studying  finitely generated noncommutative algebras and their modules.  This invariant measures the asymptotic  growth rate of an algebra or a module. For  the universal enveloping algebra $U(\mathfrak{g})$ of a finite-dimensional complex Lie algebra $\mathfrak{g}$, it gives the dimension of $\mathfrak{g}$.  One can also consider its dual algebra $\mathcal{O}(G)$: the  algebra  of regular functions   on  a  compact Lie group $G$ generated by matrix co-efficients of finite-dimensional  unitary representations.  Banica and Vergnioux \cite{BanVer-2009aa} proved that for a connected simply connected compact real Lie group $G$, GKdim of the Hopf algebra $\mathcal{O}(G)$   is equal to the manifold dimension of $G$. A  natural question arises: does the same result hold for the $q$-deformation of $G$, and its homogeneous spaces?  For  the  $q$-deformation  $G_q$ of 
	a linearly reductive group $G$,  it turns out to be true because GK dim of a finitely generated cosemisimple Hopf algebra is determined  by its representation ring and the dimension function, which are isomorphic for $G$ and $G_q$ (for detail, see \cite{BanVer-2009aa}, \cite{ChiWalWan-2019aa}). 
	However, this method  does not apply to modules and quantum homogeneous spaces, and one needs to take an entirely different approach to tackle the computation for these structures.  Chakraborty and Saurabh (\cite{ChaSau-2018aa}, \cite{ChaSau-2019aa})  took up the case of $q$-deformation of a classical Lie group of type $A, C$, and $D$ and computed the invariant for its modules and  the function algebra of a large class of homogeneous spaces. In the present article, we exhibit these computations for the case of type $B$-modules and type $B$-homogeneous spaces and thus complete the investigation for classical linear groups.

	Let $G$ denote a   classical  linear group of rank $n$,
	and $\mathfrak{g}$ be its complexified Lie algebra. For $q \in (0,1)$, let $G_q$ be the $q$-deformation of $G$, and $\mathcal{O}(G_q)$ be the algebra of regular functions on $G_q$. Fix a nondegenerate symmetric ad-invariant form $\langle \cdot, \cdot \rangle$ on $\mathfrak{g}$ such that its restriction to the real Lie algebra $g$ of $G$ is negative definite. Let $\Pi:= \{ \alpha_1, \alpha_2, \cdots \alpha_n\}$ be the set of simple roots.  The Weyl group $W_n$ of $G$ can be described as the group generated by the reflections $\{s_i:1\leq i \leq n\}$, where $s_i$ is the reflection defined by the root $\alpha_i$.  
	Given an element $w$ in the Weyl group and an element $t$ in a fixed maximal torus of $G$, Korogodski and Soibelman (\cite{KorSoi-1998aa}) constructed a simple unitarizable $\mathcal{O}(G_q)$-module $V_{t,w}$, and proved that these are precisely all 
	upto equivalence. 
	Chakraborty and Saurabh (\cite{ChaSau-2018aa}) proved that the Gelfand Kirillov dimension of $V_{t,w}$ is equal to length $\ell(w)$ of the element $w$ of $W_n$  if $G$ is of type $A, C$ or $D$.  It is natural to expect that the same result holds for the family of type $B$ quantum groups. The first main result of this paper says that this is indeed the case. 
	\bthm
	Let $w \in W_n$ and $t \in \bbbt^n$. If $V_{t,w}$ is the simple unitarizable left $\mathcal{O}(SO_q(2n+1))$-module associated to $w$ and $t$, then  one has the following: $$\text{GKdim}(V_{t,w})=\ell(w).$$
	\ethm
	\noindent We give a brief sketch of the proof. 
	For $R\subset \Pi$, let $W_{n,R}$   be the subgroup of $W_n$ generated by the simple reflections $s_{\alpha}$ with $\alpha \in R$. Define  $W_n^R$ to be a subset of $W_n$ consisting of
	reflections $w$ such that $w(\alpha)$ is a positive root of $\mathfrak{g}$ for all $\alpha \in R$. 
	Choose $R_k=\{\alpha_j :n-k+1 \leq j \leq n\}$ for $1\leq k \leq n$. One can then write  $w=w_1w_2\cdots w_{n-1}w_n$, where $w_1w_2\cdots w_k \in W_{n,R_k}$ and 
	$w_{k+1}\cdots w_n \in W_n^{R_k}$ for all $1\leq k \leq n$. Next, we attach a diagram $\mathcal{D}_{t,w}$ to each module $V_{t,w}$ in a way such that all the paths from one node to another determine the image of the canonical generators of  $\mathcal{O}(SO_q(2n+1))$.   We define a notion of embeddability of diagrams, and show that  the diagram  $\mathcal{D}_{t,w_1w_2\cdots w_k}$ is embeddable  in the diagram $\mathcal{D}_{t,w}$ for $1\leq k \leq n-1$. This allows us to use the generators of  $\mathcal{O}(SO_q(2i+1))$ for $ i \leq n$ for the computation of the invariant.  Employing this  and step-wise  induction, we prove the claim. The same method works for the quantized  algebra of regular functions defined on certain homogeneous spaces of $SO_q(2n+1)$. More precisely, we have the following result. 
	\begin{thm}
		Let $w_n^{R_k}$ be the longest element of the Weyl group $W^{R_m}_n$ of $SO(2n+1)$. Then we have 
		$$\text{GKdim}~(\mathcal{O}((SO_q(2n+1)/K_q^{R_m,P(R^c_m)}))= 2\ell(w_n^{R_m})+n-m+1=\text{dim}~(SO(2n+1)/K^{R_m,P(R^c_m)}).$$
	\end{thm}
	We prove  analog of these results  for $\mathcal{O}(\mbox{Spin}_{q^{1/2}}(2n+1))$-modules and for quotient algebras of $\mathcal{O}(\mbox{Spin}_{q^{1/2}}(2n+1))$.  One crucial advantage of our approach is that  it is algorithmic in nature, and therefore one can expect the method working for modules of a more general class of quantum groups  and their homogeneous spaces.

	The paper is organised as follows. In the next section, we recall the definition of the quantized algebra of regular functions $\mathcal{O}(SO_q(2n + 1))$ defined on the compact quantum group $SO_q(2n + 1)$ using FRT approach (\cite{KliSch-1997aa}). By applying the pairing between 
	$U_{q^{1/2}}(so(2n + 1))$ and $\mathcal{O}(SO_q(2n + 1))$, we obtain a family of simple unitarizable left $\mathcal{O}(SO_q(2n + 1))$-modules and $\mathcal{O}(\mbox{Spin}_{q^{1/2}}(2n + 1))$-modules. 
	Section $2$ computes the GK dim of such modules. In section $3$, we compute the GK dim of the quantized algebra of regular functions defined on certain quotient spaces of $SO_q(2n + 1)$ and $\mbox{Spin}_{q^{1/2}}(2n+1)$.

	Here we set-up some notations which will be used throughout the paper. The letter $q$ denote a real number in the interval $(0,1)$. Elements of  Weyl groups are called Weyl words. We denote by $\ell(w)$ the length of a Weyl word $w$. The symbol  $\mathbb{T}$ is reserved for the set of complex numbers with modules equal to $1$. 
	The standard linear basis of  $c_{00}(\mathbb{N})$ and $c_{00}(\mathbb{Z})$ is denoted by $\{e_{n}~: n\in\mathbb{N}\}$ and $\{e_{n}~: n\in\mathbb{Z}\}$, respectively.  
	The number operator $e_n\mapsto ne_n$ is denoted by $N$. The letter $S$ is for the left shift operator  $e_n\longmapsto e_{n-1}$. We denote by $\overleftarrow{\prod_{j=1}^n} b_j$ 	 the product $b_nb_{n-1}\cdots b_1$.  For
	$m,n\in \bbn$,   the $q$-binomial coefficient $\binom{n}{m}_q$ is given by  $\frac{[n]_q}{[m]_q [n-m]_q}$, where $[n]_q=\frac{q^n-q^{-n}}{q-q^{-1}}$. We denote by  $C$  a generic nonzero constant. Let $T_1$ and $T_2$ be two endomorphism of $c_{00}(\mathbb{N})^{\otimes k}$ and let  $W$ be a subspace of $c_{00}(\mathbb{N})^{\otimes k}$. We say that $ T_1\sim T_2$ on $W$ if there exits positive integers $n_1,n_2,\cdots,n_k$ such that
$$T_1=CT_2 (q^{n_1N}\otimes q^{{n_2}N}\otimes \cdots\otimes q^{{n_k}N}) $$
on W.

	\newsection{Preliminaries }
	Following \cite{KliSch-1997aa}, we recall the Hopf $\ast$-algebra structure of $\mathcal{O} (SO_q(2n+1))$  and the pairing between $U_{q^{1/2}}(so_{2n+1})$ and $\mathcal{O}(SO_{q}(2n+1))$. 
	\subsection{The Hopf $\ast$-algebra $\mathcal{O} (SO_q(2n+1))$}

	Let us begin by introducing some important  notations:
	\begin{IEEEeqnarray}{rCl} 
		i^{'} & = & 2n+2-i, \nonumber ~~~
		\rho_{i}  =  (2n+1)/2-i \quad \text{ if } i \leq i^{'}, ~~\rho_{i^{'}} = -\rho_{i},  \, \quad 
		D_{j}^{i} =  \delta_{ij} q^{-\rho_{i}}, \nonumber\\
		R_{mn}^{ij}& = &\begin{cases}
			q^{\delta_{ij}-\delta_{ij^{'}}}\delta_{im}\delta_{jn} + (q-q^{-1})(\delta_{jm}\delta_{in}- D_{i}^{j}D_{n}^{m}) & \mbox{ if } i>m, \cr
			q^{\delta_{ij}-\delta_{ij^{'}}}\delta_{im}\delta_{jn}  & \mbox{ if } i\leq m, \cr
		\end{cases}. \nonumber
	\end{IEEEeqnarray}
	Here $\delta_{ij}$ is the Kronecker delta function. Let $A(R)$ be the unital associative algebra generated by  generated by $v_{l}^{k}$, $k,l=1,2,\cdots, 2n+1$ satisfying 
	the following relations:
	\begin{IEEEeqnarray}{rCl} \label{2.1}
		\sum_{k,l=1}^{2n+1}R_{kl}^{ji}v_{s}^{k}v_{t}^{l}-R_{st}^{lk}v_{k}^{i}v_{l}^{j}=0,
		\quad i,j,s,t=1,2,\cdots ,2n+1. \label{relations}
	\end{IEEEeqnarray}
	The matrices $(\!(v_{l}^{k})\!)$
	and $(\!(D_{j}^{i})\!)$ are denoted as $V$ and $D$, respectively. Define $J$ to be the two sided ideal generated by entries of the matrices $VDV^{t}D^{-1}-I$ and $DV^{t}D^{-1}V-I$. Let $\mathcal{O}(SO_{q}(2n+1))$ denote the quotient 
	algebra $A(R)/J$.  The  Hopf $\ast$-algebra structure on  $\mathcal{O}(SO_{q}(2n+1))$ comes from the following maps.
	\begin{align*}
		&\bullet\text{Comultiplication}
		:\Delta(v_{l}^{k}) = \sum_{i=1}^{N}v_{i}^{k} \otimes v_{l}^{i},~~~~~\bullet\text{Counit}:
		\epsilon(v_{l}^{k}) = \delta_{kl}, \\& \bullet\text{Antipode}:
		S(v_{l}^{k}) = \epsilon_{k}\epsilon_{l}q^{\rho_{k}-\rho_{l}}v_{l^{'}}^{k^{'}}, ~~~~~~~~~~~~~~\bullet\text{Involution}:
		(v_{l}^{k})^{*} = \epsilon_{k}\epsilon_{l}q^{\rho_{k}-\rho_{l}}v_{l^{'}}^{k^{'}}.
	\end{align*}

	\subsection{Pairing between $U_{q^{1/2}}(so_{2n+1})$ and $\mathcal{O}(SO_{q}(2n+1))$}  
	We recall from \cite{KliSch-1997aa} the Hopf algebra pairing between $U_{q^{1/2}}(so_{2n+1})$ and $\mathcal{O}(SO_{q}(2n+1))$  and  between $U_q(sl_2)$ and $\mathcal{O}(SL_q(2))$.
	\bthm \cite{KliSch-1997aa} \label{pairing}
	There exists a unique  dual pairing $\left\langle\cdot,\cdot\right\rangle$ between the Hopf algebras $U_q(sl_2)$ and $\mathcal{O}(SL_q(2))$, as well as $U_{q^{1/2}}(so_{2n+1})$ and $\mathcal{O}(SO_{q}(2n+1))$. This pairing  is given by the following equation:
	\[
	\left\langle f,v_{l}^{k}\right\rangle = t_{kl}(f),        \hspace{0.5in} \mbox{for }  \,k,l=1,2,\cdots,2n+1,\\
	\]
	where $t_{kl}$ is the matrix element of $T_{1}$, which corresponds to the vector representation of $U_{q}(sl_{2})$ in first case and $U_{q^{1/2}}(so_{2n+1})$ in second case.  Moreover, the pairing $\left\langle\cdot,\cdot\right\rangle$  is nondegenerate. 
	\ethm
	We will  describe $T_{1}$ for both Hopf $\ast$-algebras. Define $E_{ij}$ as a $2n \times 2n$ 
	matrix with $1$ in the $(i,j)^{th}$ position and $0$ elsewhere, and $D_{j}$ as a diagonal matrix with $q$ in the $(j,j)^{th}$ 
	position and $1$ elsewhere on the diagonal. For the quantum universal enveloping algebra  $U_{q}(so_{2n+1})$, we have
	
	\[ \left.
	\begin{array}{rcl}
		T_1(K_i) &=& D_i^{-1}D_{i+1}D_{2n-i+1}^{-1}D_{2n-i+2}, \nonumber\\
		T_1(E_i) &=& E_{i+1,i}-E_{2n-i+2,2n-i+1},\nonumber\\
		T_1(F_i) &=& E_{i,i+1}-E_{2n-i+1,2n-i+2},\nonumber
	\end{array} \right\}\quad  \mbox{for} \quad i \in \{1,2,\cdots,n-1\},
	\]
	and for $i=n$,
	\[ 
	T_1(K_n) = D_n^{-1}D_{n+2},~~~~
	T_1(E_n) = c(E_{n+1,n}-q^{1/2}E_{n+2,n+1}),~~~~
	T_1(F_n) = c(E_{n,n+1}-q^{-1/2}E_{n+1,n+2}),
	\]
	where $c=(q^{1/2}+q^{-1/2})^{1/2}$.
	For $U_{q}(sl_{2})$, one has  the following:\\
	\begin{IEEEeqnarray}{rCl}
		T_1(K) =   \left({\begin{matrix}
				q^{-1} & 0\\
				0& q\\
		\end{matrix} } \right),\qquad
		T_1(E) =  \left( {\begin{matrix}
				0 & 0\\
				1& 0\\
		\end{matrix} } \right), \qquad
		T_1(F) =  \left( {\begin{matrix}
				0 & 1\\
				0& 0\\
		\end{matrix} } \right).\nonumber
	\end{IEEEeqnarray}
	
	\noindent	\textbf{$U_{q^{1/2}}(so_{2n+1})$-module structure}: The pairing $\left\langle\cdot,\cdot\right\rangle$  induces a $U_{q^{1/2}}(so_{2n+1})$-module structure on  $\mathcal{O}(SO_{q}(2n+1))$ which is as follows:
	\begin{IEEEeqnarray}{rCl} \label{module}
	fv=(1\otimes \left\langle\, f\,,\cdot\,\right\rangle )\Delta(v)= \left\langle f,  v_{(2)}\right\rangle v_{(1)} ; \, \mbox { for } a \in \mathcal{O}(SO_{q}(2n+1)), \,f \in U_{q^{1/2}}(so_{2n+1}),
	\end{IEEEeqnarray}
	where $\Delta(v)=\sum \, v_{(1)} \otimes v_{(2)}$ in Sweedler notation. One can extend the action of $K_i$'s to the maximal torus $\mathbb{T}^n$ as follows:
	Let $V$ be a finite-dimensional admissible $U_{q^{1/2}}(so_{2n+1})$-module and   $\lambda$ be an integral weight.  Denote by $V(\lambda)$ the space of vectors $v\in V$ having weight $\lambda$.  Define an action of the maximal torus $\mathbb{T}^n$  on $V(\lambda)$ as follows:
	$$(t_1,t_2,\cdots ,t_n)v=(\prod_{j=1}^n t_i^{(\lambda ,\alpha)})v.$$
	This action can be extended to $V$ as $V=\oplus_{\lambda } V(\lambda)$.   Let $\mu \in P_{+}$ be a dominant integral weight and $V_{\mu}$ be a simple unitarizable module  $U_{q^{1/2}}(so_{2n+1})$-module with $\mu$  as the highest weight. Let $\left( \cdot, \cdot \right)$ be the inner product on $V_{\mu}$. For $\omega, \upsilon \in V_{\mu}$, denote the  matrix entry $\left( \omega, \odot \upsilon \right)$ by $ C_{\omega, \upsilon}^{\mu}$.  The action of the maximal torus  $\bbbt^n$ on  $C_{\omega, \upsilon}^{\mu}$ of $V_{\mu}$ is given by:
	\begin{IEEEeqnarray}{rCl}\label{torus}
		(t_1,t_2,\cdots ,t_n)C_{\omega, \upsilon}^{\mu}= \left( \omega\, , \,	\odot (t_1,t_2,\cdots ,t_n) \upsilon\right). 
	\end{IEEEeqnarray}

	\newsection{Simple unitarizable modules of  $\mathcal{O}(SO_q(2n+1))$}
	In this section, we will describe all simple unitarizable $\mathcal{O}(SO_{q}(2n+1))$-modules and associate a diagram to each such module.  \\
	
	\noindent\textbf{Elementary representation of $\mathcal{O}(SO_{q}(2n+1))$:} Fix $i \in \{1,2,\cdots,n\}$. Define $q_i=q^{d_i}$, where $d_i=2$ for $i=1,2,\cdots ,n-1$ and $d_n=1$. Let $K$, $F$, and $E$ be the standard generators of $U_{q_i}(sl_2)$.  Let  $\varphi_i : U_{q_i}(sl_2) \longrightarrow U_{q^{1/2}}(so_{2n+1})$ be a homomorphism given on generators by,

	$$\varphi_i(K)=K_i,~~\varphi_i(E)=E_i,~~\varphi_i(F)=F_i .$$
	By duality, we get a surjective homomorphism
	\begin{displaymath}
		\varphi_i^* :\mathcal{O}(SO_{q}(2n+1)) \longrightarrow \mathcal{O}(SU_{q_i}(2))
	\end{displaymath}
	given by
	\begin{displaymath}
		\left\langle f, \varphi_i^*(v_l^k)\right\rangle = \left\langle \varphi_i(f), v_l^k\right\rangle.
	\end{displaymath}
	In particular,
	\begin{IEEEeqnarray}{rCl}
		\left\langle K, \varphi_i^*(v_l^k)\right\rangle = \left\langle K_i, v_l^k\right\rangle, \quad
		\left\langle E, \varphi_i^*(v_l^k)\right\rangle = \left\langle E_i, v_l^k\right\rangle, \quad 
		\left\langle F, \varphi_i^*(v_l^k)\right\rangle = \left\langle F_i, v_l^k\right\rangle. \label{r}
	\end{IEEEeqnarray}
	Let $\left( \begin{smallmatrix}
		\alpha & -q\beta^*\\
		\beta & \alpha^*
	\end{smallmatrix}\right)$ be the fundamental co-representation of $\mathcal{O}(SU_{q_i}(2))$. Note that  $\alpha$ and $\beta$ generate the algebra $\mathcal{O}(SU_{q_i}(2))$.  
	\bthm \label{homomorphism}
	For $i \in \{1,2,\cdots, n-1\}$, one has the following:
	\[
	\phi_i^*(v_l^k)=\begin{cases} 
		\alpha  & \mbox{ if } (k,l)=(i,i) \mbox{ or } (2n-i+1,2n-i+1),\cr
		\alpha^*& \mbox{ if } (k,l)=(i+1,i+1) \mbox{ or } (2n-i+2,2n-i+2),\cr
		-q^2\beta^* & \mbox{ if } (k,l)=(i,i+1),\cr
		\beta & \mbox{ if } (k,l)=(i+1,i),\cr
		-\beta & \mbox{ if } (k,l)=(2n-i+2,2n-i+1),\cr
		q^2\beta^* & \mbox{ if } (k,l)=(2n-i+1,2n-i+2),\cr
		\delta_{kl} & \mbox{ otherwise }. \cr
	\end{cases}
	\]
	For $i=n$, one has  the following:
	\[
	\phi_n^*(v_l^k)=\begin{cases} 
		\alpha^2 & \mbox{ if } (k,l)=(n,n),\cr
		-q\sqrt{1+q^2}\, \beta^*\alpha & \mbox{ if } (k,l)=(n,n+1),\cr
		(q\beta^*)^2& \mbox{ if } (k,l)=(n,n+2), \cr
		\sqrt{1+q^2}\, \beta \alpha & \mbox{ if } (k,l)=(n+1,n),\cr
		1-(1+q^2)\beta^*\beta & \mbox{ if } (k,l)=(n+1,n+1),\cr
		-q\sqrt{(1+q^2)}\,\alpha^*\beta^* & \mbox{ if } (k,l)=(n+1,n+2),\cr
		\beta^2 & \mbox{ if } (k,l)=(n+2,n), \cr	
		\sqrt{(1+q^2)}\,\alpha^* \beta & \mbox{ if } (k,l)=(n+2,n+1),\cr
		(\alpha^*)^2 & \mbox{ if } (k,l)=(n+2,n+2),\cr
		\delta_{kl} & \mbox{ otherwise }. \cr     
	\end{cases}
	\]
	\ethm 
	\prf We will prove the claim for the case of $i=n, k=n$, and $l=n$; the other cases can be shown using a similar argument. For that, take $a,b,c \in \bbn$. A straightforward computation using equation (\ref{r}) shows that 
	\[
	\left\langle K^aF^bE^c, \varphi_n^*(v_n^n)\right\rangle=\left\langle K^aF^bE^c, \alpha^2\right\rangle.
	\]
	Since the pairing is nondegenrate, we get the claim.
	\qed \\
	Let $\pi$ be a representation of $\mathcal{O}(SU_{q}(2))$ on  $c_{00}(\bbn)$ given by;
	\begin{IEEEeqnarray}{rCl}
		\pi(\alpha)=  \sqrt{1-q^{2N+2}}S , \, \quad \pi(\beta)=	q^N. \label{3.3}
	\end{IEEEeqnarray}
	Define $\pi_{s_{i}} = \pi \circ \varphi_{i}^{*}.$ Applying $(\ref{r})$, we have, for $i=1,2,\cdots,n-1$,
	\[
	\pi_{s_{i}}(v_l^k)=\begin{cases}
		\sqrt{1-q^{4N+4}}S & \mbox{ if } (k,l)=(i,i) \mbox{ or } (2n-i+1,2n-i+1),\cr
		S^*\sqrt{1-q^{4N+4}} & \mbox{ if } (k,l)=(i+1,i+1) \mbox{ or } (2n-i+2,2n-i+2),\cr
		-q^{2N+2} & \mbox{ if } (k,l)=(i,i+1),\cr
		q^{2N} & \mbox{ if } (k,l)=(i+1,i),\cr
		-q^{2N} & \mbox{ if } (k,l)=(2n-i+2,2n-i+1),\cr
		q^{2N+2} & \mbox{ if } (k,l)=(2n-i+1,2n-i+2),\cr
		\delta_{kl} & \mbox{ otherwise }. \cr
	\end{cases}
	\]
	For $i=n$,
	\[
	\pi_{s_n}(v_l^k)=\begin{cases}
		\sqrt{(1-q^{2N+2})(1-q^{2N+4})}S^{2} & \mbox{ if } (k,l)=(n,n),\cr
		I-(1+q^2)q^{2N} & \mbox{ if } (k,l)=(n+1,n+1),\cr
		{S^{*}}^{2}\sqrt{(1-q^{2N+4})(1-q^{2N+2})} & \mbox{ if } (k,l)=(n+2,n+2),\cr
		q^{N}\sqrt{(1+q^2)(1-q^{2N+2})}S & \mbox{ if } (k,l)=(n+1,n),\cr
		q^{N+1}\sqrt{(1+q^2)(1-q^{2N+2})}S & \mbox{ if } (k,l)=(n,n+1),\cr
		
		S^{*}\sqrt{(1+q^2)(1-q^{2N+2})}q^{N} & \mbox{ if } (k,l)=(n+2,n+1),\cr
		-S^{*}\sqrt{(1+q^2)(1-q^{2N+2})}q^{N+1} & \mbox{ if } (k,l)=(n+1,n+2),\cr

		q^{2N+2} & \mbox{ if } (k,l)=(n,n+2), \cr
		q^{2N} & \mbox{ if } (k,l)=(n+2,n), \cr	
		\delta_{kl} & \mbox{ otherwise }. \cr     
	\end{cases}
	\]
	\noindent This makes $c_{00}(\bbn)$ a $\mathcal{O}(SO_q(2n+1))$-module, where the action is given by the elementary representation  $\pi_{s_{i}}$. We denote this module by $V_{s_i}$.
	For $t=(t_{1},t_{2},\cdots ,t_{n}) \in \bbbt^{n}$, there is a one dimensional $\mathcal{O}(SO_q(2n+1))$-module $V_{t}$ with action coming from the map  $\tau_t:\mathcal{O}(SO_q(2n+1)) \longrightarrow \bbc $ given by
	\[
	\tau_{t}(v_l^k)=\begin{cases}
		\overline{t_{k}}\delta_{kl} & \mbox{ if } k<n+1,\cr
		\delta_{kl} & \mbox{ if } k=n+1,\cr
		t_{2n+1-k}\delta_{ij} & \mbox{ if } k> n+1.\cr
	\end{cases}
	\]
	Given two actions $\eta$ and $\gamma$ of $\mathcal{O}(SO_q(2n+1))$, define an action  $\eta * \gamma := (\eta \otimes \gamma)\circ \Delta$. For any two $\mathcal{O}(SO_q(2n+1))$-module $V_{\eta}$ and $V_{\gamma}$,  the  $\mathcal{O}(SO_q(2n+1))$-module $V_{\eta}\ast V_{\gamma}$ is defined as the vector space $V_{\eta} \otimes  V_{\gamma}$ with   $\mathcal{O}(SO_q(2n+1))$-action  coming from  $\eta * \gamma$.
	\bdfn Let  $w \in W_n$ and $t \in \bbbt^n$. Let $s_{i_{1}}s_{i_{2}}...s_{i_{k}}$ be a reduced expression for $w$.  
	Then  the isomorphism class of the module $V_t\ast V_{s_{i_1}}\ast V_{s_{i_2}}\ast \cdots \ast V_{s_{i_k}}$  does not depend on   reduced expressions of $w$.  We denote this module by $V_{t,w}$ and  the corresponding action  by $\pi_{t,w}^n$. It is a simple unitarizable module of $\mathcal{O}(SO_q(2n+1))$. If $t=1$, then $V_{1,w}=V_1\ast V_{s_{i_1}}\ast V_{s_{i_2}}\ast \cdots \ast V_{s_{i_k}}= V_{s_{i_1}}\ast V_{s_{i_2}}\ast \cdots \ast V_{s_{i_k}}$. We will omit $1$, and denote it by $V_w$. 
	\edfn
		The matrix entries  $(\!(x_{j}^{i})\!)$ of the representation of  $U_{q^{1/2}}(so_{2n+1})$  with highest weight $(\frac{1}{2}, \frac{1}{2}, \cdots , \frac{1}{2})$ generates the Hopf-$\ast$-algebra $\mathcal{O}(\mbox{Spin}_q(2n+1))$.  The corepresentation of $\mathcal{O}(\mbox{Spin}_q(2n+1))$ corresponding to the highest weight $(1,0,\cdots ,0)$ is the same as $(\!(v_{l}^{k})\!)$ whose entries generates a proper Hopf-$\ast$-subalgebra  $\mathcal{O}(SO_q(2n+1))$. 
	In the same way as above, one can construct a simple unitarizable left $\mathcal{O}(\mbox{Spin}_q(2n+1))$-module 
	$V_{t,w}^{\mbox{Spin}}$ for each $t \in \bbbt^{n}, w \in W_n$.  The  following theorem says that they are all upto isomorphism.

	\bthm \cite{KorSoi-1998aa} \label{2.3} The set
	$$\{V_{t,w}^{\mbox{Spin}}: t \in \bbbt^{n}, w \in W_n \}$$
	forms a complete set of mutually inequivalent simple unitarizable left $\mathcal{O}(\mbox{Spin}_q(2n+1))$-modules.
	\ethm
	
	\brmrk \label{restriction}
	Note that if we  restrict the algebra action to $\mathcal{O}(SO_q(2n+1))$, and  view the module  $V_{t,w}^{\mbox{Spin}}$ as $\mathcal{O}(SO_q(2n+1))$-module then 
	it is same as $V_{t,w}$. 
	\ermrk
	
	\subsection{Decomposition of a Weyl word}
	We recollect from \cite{Hum-1990aa} the following facts about the Weyl group $W_n$ of $\mathfrak{so}(2n+1)$. For a detailed treatment, we refer the reader to \cite{Hum-1990aa}.
	\bppsn $\cite{Hum-1990aa}$\label{3.2}
	Let $W_n$ be the Weyl group of the Lie algebra $\mathfrak{so}(2n+1)$. 
	\begin{enumerate}[(i)]
		\item
		Any element $w \in W_{n}$ has a reduced expression of the form:\, $\psi_{n,k_n}^{(\epsilon_{n})}\psi_{n-1,k_{n-1}}^{(\epsilon_{n-1})}\cdots \psi_{1,k_1}^{(\epsilon_{1})}$,  
		where 
		$\epsilon_{r} \in \left\{0,1,2\right\}$ and $r\leq k_{r}\leq n$ with the convention that,
		\[
		\psi_{r,k_{r}}^{\epsilon}=\begin{cases}
			s_{k_{r}-1}s_{k_{r}-2}...s_{r} & \mbox{ if } \epsilon=1,\cr
			s_{k_{r}}s_{k_{r}+1}\cdots...s_{n-1}s_{n}s_{n-1}\cdots s_{k_{r}}s_{k_{r}-1}\cdots s_{r} & \mbox{ if } \epsilon=2,\cr
			\mbox{ empty string } & \mbox{ if } \epsilon=0,\cr
		\end{cases}
		\]
		\item
		The composition $\psi_{n,n}^2\psi_{n-1,n-1}^2\cdots \psi_{1,1}^2$ is a  reduced expression  for the longest word of $W_n$.
	\end{enumerate}
	\eppsn\label{2.4aa}

	\noindent For $R\subset \Pi$, let $W_{n,R}$  to be the subgroup of $W_n$ generated by the simple reflections $r_{\alpha}$ with $\alpha \in R$.  Let $W_n^R$ be a subset of $W_n$ consisting of
	reflections $w$ such that $w(\alpha)$ is a positive root of $\mathfrak{so}(2n+1)$ for all $\alpha \in R$. Then every element $w$ in $W_n$ can be decomposed uniquely as $w = w^{\prime}w^{\prime \prime}$ with $w^{\prime}\in W_{n,R}$ and $w^{\prime \prime} \in W_n^R$ such that $\ell(w) = \ell(w^{\prime}) + \ell(w^{\prime \prime})$.   For $ 1\leq k \leq n$, define
	$$R_k=\{\alpha_j :n-k+1 \leq j \leq n\}.$$ One can write  $w=w_1w_2\cdots w_{n-1}w_n,$ where $w_1w_2\cdots w_k \in W_{n,R_k}$ and 
	$w_{k+1}\cdots w_n \in W_n^{R_k}$ for all $1\leq k \leq n$. The above proposition says that  $$w_r=\psi_{n-r+1,k_{n-r+1}}^{(\epsilon_{n-r+1})}.$$ We call $w_r$ the $r$-th part of $w$. 
	\brmrk Note that $W_{n,R_k}$ is naturally isomorphic to the Weyl group $W_k$ of $\mathfrak{so}(2k+1)$. With this identification, one can think of the Weyl word $w_1w_2\cdots w_k$ as an element of $W_k$.
	\ermrk
	
	\subsection{Diagram representation }\label{2.4}
	
		In what follows, we will associate a diagram  to each representation of $\mathcal{O}(SO_q(2n+1))$ in the same way  as given in (\cite{ChaPal-2008aa}, \cite{ChaSau-2018aa}).  These diagrams will have $2n+1$ nodes on both sides indexed by $\{1,2,\cdots ,2n+1\}$.  There are certain edges from left to right side nodes,  each representing an endomorphism of the vector space mentioned on the top of the diagram.   
	The following table describes the endomorphism that an edge represents. 
	\begin{center}
		\begin{tabular}{|c c|c c|c c|}
			\hline Arrow type & Operator & Arrow type & Operator & Arrow type & Operator  \\
			
			\hline  \xymatrix@C=20pt@R=15pt{\ar@{-}[r]&\\} & $I$ & {\def\labelstyle{\scriptscriptstyle}\xymatrix@C=20pt@R=15pt{\ar@{.>}[dr]&\\
					&} }&	$q^{N}\sqrt{(1+q^2)(1-q^{2N+2})}S $ & \xymatrix@C=20pt@R=15pt{{}&\\ \ar@{-}[ur]_{+}} &$-q^{2N+2}$\\

			\hline \xymatrix@C=20pt@R=15pt{\ar@{-}[r]_{-}&\\}& $ \sqrt{1-q^{4N+4}}S$ &\xymatrix@C=20pt@R=15pt{\ar@{-}[r]_{--}&\\}  & $\sqrt{(1-q^{2N+2})(1-q^{2N+4})}S^{2} $&{\def\labelstyle{\scriptscriptstyle}\xymatrix@C=20pt@R=15pt{\ar@{-}[dr]_{+}&\\
					&} }& $q^{2N}$\\

			\hline  \xymatrix@C=20pt@R=15pt{\ar@{--}[r]&\\}   & $I-(1+q^2)q^{2N}$ &\xymatrix@C=20pt@R=15pt{\ar@{-}[r]^{++}&\\}	&${S^{*}}^{2}\sqrt{(1-q^{2N+4})(1-q^{2N+2})}$&{\def\labelstyle{\scriptscriptstyle}\xymatrix@C=20pt@R=15pt{\ar@{-}[dr]_{-}&\\
					&} }&$-q^{2N}$ \\

			\hline \xymatrix@C=20pt@R=15pt{{}&\\ \ar@{-}[ur]_{-}}  & $q^{2N+2}$ & \xymatrix@C=20pt@R=15pt{{}&\\ \ar@{-->}[ur]}   & $	q^{N+1}\sqrt{(1+q^2)(1-q^{2N+2})}S$& &\\

			\hline \xymatrix@C=20pt@R=15pt{\ar@{-}[r]^{+}&\\} & $S^* \sqrt{1-q^{4 N+4}}$& {\def\labelstyle{\scriptscriptstyle}\xymatrix@C=20pt@R=15pt{\ar@{->>}[dr]&\\
					&} }   & 	$S^{*}\sqrt{(1+q^2)(1-q^{2N+2})}q^{N}$
			& {\def\labelstyle{\scriptscriptstyle}\xymatrix@C=20pt@R=15pt{\ar@{-}[dr]&\\
					&} }   & $q^{2N}$ \\

			\hline  \xymatrix@C=20pt@R=15pt{{}&\\ \ar@{~>}[ur]} & $q^{2N+2}$& \xymatrix@C=20pt@R=15pt{{}&\\ \ar@{=}[ur]}  & $-S^{*}\sqrt{q(1+q)(1-q^{2N+2})}q^{N}$& \xymatrix@C=20pt@R=15pt{\ar@{-}[r]^{t}&\\}& $M_t$\\ \hline
		\end{tabular}
	\end{center}

	Let us first consider the diagram $D_{s_i}$ for the elementary representation $\pi_{s_i}$. 
	Here the underlying vector space is $c_{00}(\bbn)$. The endomorphism  represented by  the edge from the left node $l$ to the right node $m$ gives the image of  $v_m^l$ under the map $\pi_{s_i}$.  If there are no edge, then $\pi_{s_i}(v_m^l)=0$. 
	Below are the diagrams for each elementary representation.  For one dimensional representation $\tau_t,\,  t \in \bbbt^n$, the same description applies, except the fact that the underlying vector space is $\bbc$.   
	One can see that $\pi_{s_i}(v_1^1)$ is the identity operator $I$, $\pi_{s_i}(v_1^3)$ equals zero operator, and $\pi_{s_i}(v_i^{i+1})$ is  $q^{N}$ if $i> 1$.

	\begin{tabular}{p{150pt}p{150pt}p{150pt}}
		
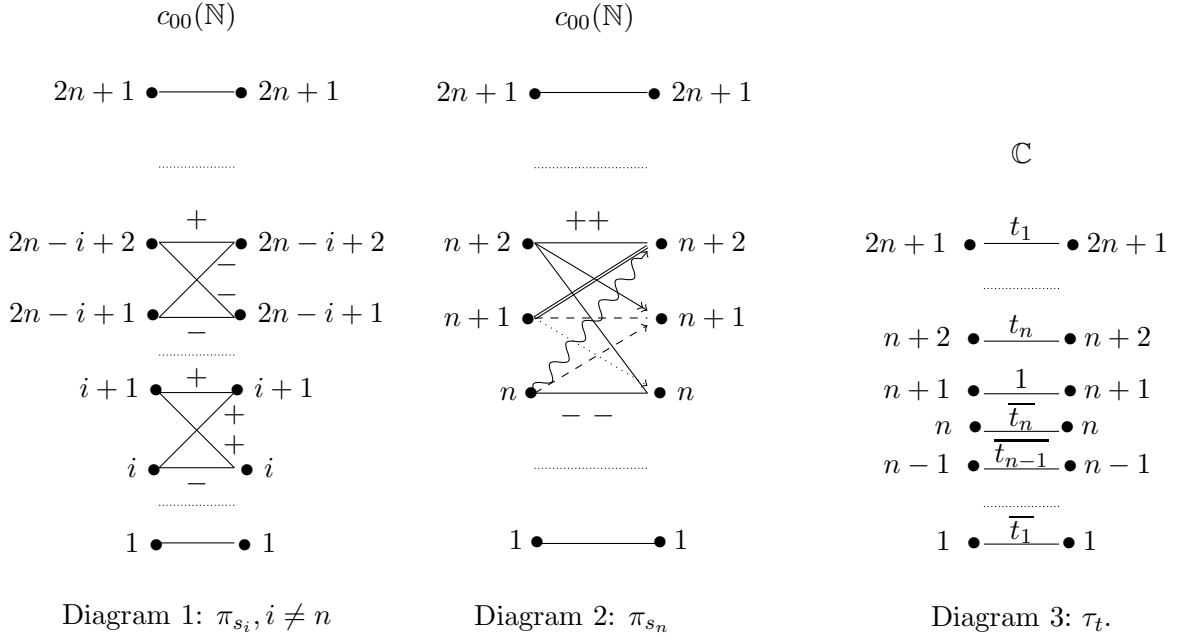
		\begin{tikzpicture}[scale=1.0, shift={(0,2)}]
			
			\draw [-] (0,6) -- (1,6);
			\draw [densely dotted] (0,5) -- (1,5);
			\draw [-] (0,4) -- (1,4);
			\draw [-] (0,4) -- (1,3);
			\draw [-] (0,3) -- (1,3);
			\draw [-] (0,3) -- (1,4);
			\draw [densely dotted] (0,2.5) -- (1,2.5);
			\draw [-] (0,2) -- (1,2);
			\draw [-] (0,2) -- (1,1);
			\draw [-] (0,1) -- (1,1);
			\draw [-] (0,1) -- (1,2);
			\draw [densely dotted] (0,0.5) -- (1,0.5);
			\draw [-] (0,0) -- (1,0);
			\node at (.5,4.3) {+};
			\node at (.9,3.7)   {$-$};
			\node at (.9,3.3)   {$-$};
			\node at (.5,2.8) {$-$};
			\node at (.5,2.2) {+};
			\node at (.5,0.79) {$-$};
			\node at (1,1.7) {+};
			\node at (1,1.3) {+};
			\node at (.5,7){$c_{00}(\bbn)$};
			\node at (-.7,6){$2n+1~\bullet$};
			\node at (1.7,6){$\bullet ~2n+1$};
			\node at (-1,4){$2n-i+2~\bullet$};
			\node at (2,4){$\bullet ~2n-i+2$};
			\node at (-.99,3.05){$2n-i+1~\bullet$};
			\node at (1.99,3.05){$\bullet ~2n-i+1$};
			\node at (-.5,2.05){$i+1~\bullet$};
			\node at (1.5,2.05){$\bullet ~i+1$};
			\node at (-.2,1.00){$i~\bullet$};
			\node at (1.3,1.00){$\bullet ~i$};
			\node at (-.2,0){$1~\bullet$};
			\node at (1.3,0){$\bullet ~1$};
			\node at (.5,-1){\text Diagram 1: $\pi_{s_i}, i\neq n$};
		\end{tikzpicture}
		&
		\begin{tikzpicture}[scale=1.0, shift={(0,2)}]
			\draw [-] (0,7) -- (1.5,7);
			\draw [densely dotted] (0,6) -- (1.5,6);
			\draw [-] (0,5) -- (1.5,5);
			\draw [-,double] (0,4) -- (1.5,4.95);
			\draw [->,decorate,decoration={coil,aspect=0}] (0,3) -- (1.5,4.89);
			\draw [->>] (0,5) -- (1.5,4.1);
			\draw [dashed] (0,4) -- (1.5,4);
			\draw [->,dotted] (0,4) -- (1.5,3.1);
			\draw [->,dashed] (0,3) -- (1.5,3.9);
			\draw [-] (0,3) -- (1.5,3);
			\draw [densely dotted] (0,2) -- (1.5,2);
			\draw [-] (0,1) -- (1.6,1);
			\draw [-] (0,5) -- (1.5,3);
			\node at (0.7,5.3) {$++$};
			\node at (.7,2.7) {$-~-$};
			\node at (.8,8){$c_{00}(\bbn)$};
			\node at (-.6,7){$2n+1~\bullet$};
			\node at (2.2,7){$\bullet ~2n+1$};
			\node at (-.6,5){$n+2~\bullet$};
			\node at (2.2,5){$\bullet ~n+2$};
			\node at (-.6,4){$n+1~\bullet$};
			\node at (2.2,4){$\bullet~ n+1$};
			\node at (-0.1,3){$n\bullet ~$};
			\node at (1.85,3){$\bullet~ n$};
			\node at (0,1.05){$1\bullet ~$};
			\node at (1.7,1.05){$~\bullet1$};
			\node at (.5,0){\text Diagram 2: $\pi_{s_n}$};
		\end{tikzpicture}
		&
		\begin{tikzpicture}[scale=1.0, shift={(0,2)}]
			\draw [-] (0,5) -- (1,5);
			\draw [densely dotted] (0,4.4) -- (1,4.4);
			\draw [-] (0,3.7) -- (1,3.7);
			\draw [-] (0,3) -- (1,3);
			\draw [-] (0,2.5) -- (1,2.5);
			\draw [-] (0,2) -- (1,2);
			\draw [densely dotted] (0,1.5) -- (1,1.5);
			\draw [-] (0,1) -- (1,1);
			\node at (.5,6.2){$\bbc$};
			\node at (-.6,5.00){$2n+1~\bullet~~~$};
			\node at (1.6,5.00){$~~~\bullet 2n+1~$};
			\node at (-.45,3.75){$n+2~\bullet~~~$};
			\node at (1.47,3.75){$~~\bullet n+2$};
			\node at (-.45,3.05){$n+1~\bullet~~~$};
			\node at (1.47,3.05){$~~~\bullet n+1~$};
			\node at (-.1,2.55){$n~\bullet~~~$};
			\node at (1.1,2.55){$~~~\bullet n~$};
			\node at (-.45,2.05){$n-1~\bullet~~~$};
			\node at (1.47,2.05){$~~~\bullet n-1~$};
			\node at (-.1,1.05){$1~\bullet~~~$};
			\node at (1.1,1.05){$~~~\bullet 1~$};
			\node at (.5,5.2){$t_1$};
			\node at (.5,3.9){$t_{n}$};
			\node at (.5,3.2){$1$};
			\node at (.5,2.7){$\overline{t_n}$};
			\node at (.5,2.2){$\overline{t_{n-1}}$};
			\node at (.5,1.2){$\overline{t_1}$};
			\node at (.5,0){Diagram 3:$~\tau_t.$};
		\end{tikzpicture}
	\end{tabular}\\[1ex]	\\

	Now, suppose $\eta$ and $\gamma$ are two representations of $\mathcal{O}(SO_q(2n+1))$
	acting on the vector spaces $V_1$ and $V_2$, respectively. The diagram for $\eta \ast \gamma$ is a concatenation of  the diagram $D_{\eta}$ for $\eta$ and $D_{\gamma}$ for $\gamma$. Place the diagram  and $D_{\gamma}$  adjacent to each other, and  identify  the $i$-th  node on the right side of $D_{\eta}$ with   the $i$-th node on the left side of  $D_{\gamma}$.  Each path from one node on the left to another node on the left represents an endomorphism of $V_1\otimes V_2$. Now, $\eta \ast \gamma(v_k^l)$ will be the sum of the endomorphisms represented by each path from the left node $l$ to the right node $k$. If there is no path then it will be zero.\\
	
	To illustrate by an example, see below the  diagram $D_{w}$  for  the representation $\pi_{\omega}$ of $\mathcal{O}(SO_{q}(7))$, where $w =s_1s_2s_3s_2s_1$. Here 
	\begin{align*}
		\pi_{\omega}(v_3^1)= &~q^{2N+2}\otimes q^{2N+2}\otimes  \sqrt{(1-q^{2N+2})(1-q^{2N+4})}S^{2} \otimes S^{*}\sqrt{1-q^{4N+4}}\otimes I -\\
		&~q^{2N+2}\otimes\sqrt{1-q^{4N+4}}S\otimes I\otimes S^{*}\sqrt{1-q^{4N+4}}\otimes I.  \\ \quad
		\pi_{\omega}(v_4^4)= &~I\otimes I\otimes (I-(1+q^2)q^{2N})\otimes I \otimes I.
	\end{align*}
	\pagebreak
	\begin{figure}[h] 
		\def\labelstyle{\scriptstyle} 
		\xymatrix@C=20pt@R=25pt{
			& & & &  &{}\ar@{}[r]_{c_{00}(\mathbb{N})~\otimes}
			& {}\ar@{}[r]_{~~~c_{00}(\mathbb{N})~\otimes}& {}\ar@{}[rr]_{~~c_{00}(\mathbb{N})~~~\otimes}&& {}\ar@{}[r]_{c_{00}(\mathbb{N})~~\otimes}
			& {}\ar@{}[r]_{~~c_{00}(\mathbb{N})}&& \\
			& & & & 
			&\circ\ar@{-}[r]^{+}\ar@{-}[rd]^{\;\;-}
			&\circ\ar@{-}[r]&\circ\ar@{-}[rr]&&\circ\ar@{-}[r]
			&\circ\ar@{-}[r]^{+}\ar@{-}[rd]^{\;\;-}&\circ\\                                                   
			& & & & 
			&\circ\ar@{-}[r]_{-}\ar@{-}[ru]_{\;\;-}
			&\circ\ar@{-}[r]^{+}\ar@{-}[rd]^{\;\;-}
			&\circ\ar@{-}[rr]&&\circ\ar@{-}[r]^{+}\ar@{-}[rd]^{\;\;-}
			&\circ\ar@{-}[r]_{-}\ar@{-}[ru]_{\;\;-}&\circ&\\
			& & & & 
			&\circ\ar@{-}[r]&\circ\ar@{-}[r]_{-}\ar@{-}[ru]_{\;\;-}
			&\circ\ar@{-}[rr]^{++}\ar@{->>}[drr]\ar@{-}[ddrr]
			&&\circ\ar@{-}[r]_{-}\ar@{-}[ru]_{\;\;-}
			&\circ\ar@{-}[r]&\circ&\\
			& & & & 
			&\circ\ar@{-}[r]&\circ\ar@{-}[r]
			&\circ\ar@{--}[rr]\ar@{=}[urr]\ar@{..>}[drr]&
			&\circ\ar@{-}[r] &\circ\ar@{-}[r]&\circ&\\
			& & & & 
			&\circ\ar@{-}[r]
			&\circ\ar@{-}[r]\ar@{-}[r]^{+}\ar@{-}[rd]^{\;\;+}&\circ \ar@{-}[rr]_{--}\ar@{-->}[urr]\ar@{~>}[uurr]  &&\circ\ar@{-}[r]^{+}\ar@{-}[rd]^{\;\;+}
			&\circ\ar@{-}[r]&\circ&\\
			& & & & 
			&\circ\ar@{-}[r]^{+}\ar@{-}[rd]^{\;\;+}&\circ\ar@{-}[r]_{-}\ar@{-}[ru]_{\;\;+}
			&\circ\ar@{-}[rr]&&\circ\ar@{-}[r]_{-}\ar@{-}[ru]_{\;\;+}
			&\circ\ar@{-}[r]^{+}\ar@{-}[rd]^{\;\;+}&\circ&\\
			& & & & 
			&\circ\ar@{-}[r]_{-}\ar@{-}[ru]_{\;\;+}
			&\circ\ar@{-}[r]&\circ\ar@{-}[rr]&&\circ\ar@{-}[r]
			&\circ\ar@{-}[r]_{-}\ar@{-}[ru]_{\;\;+}&\circ&\\   
		}
		\caption{Diagram of $\pi_{\omega}$}\label{diagram1}      \end{figure}
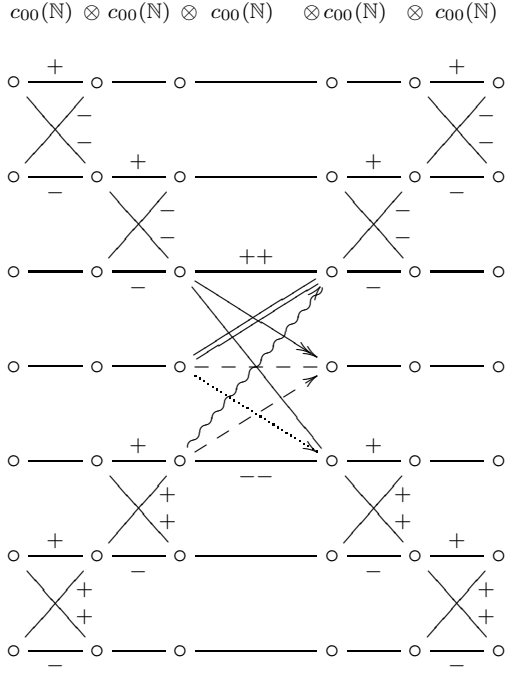

	\bdfn \label{definition of diagram embedding} \textbf{Diagram Embedding }:
	For $1\leq k \leq n-1 $, define 
	$L_n^k:=n-k+1$ and  $M_n^k:=n+k+1$.  Let $w \in W_n$. One can write   $w=w_1w_2\cdots w_n$.   
	We say that the diagram $D_{w_1w_2\cdots w_k}$ is embeddable in the diagram  $D_{w_1w_2\cdots w_{k+l}}$ if there exists a map 
	\[
	\lambda_k^{k+l}: \{L_n^k,L_n^k+1, \cdots ,M_n^k\} \longrightarrow  \{L_n^{k+l},L_n^{k+l}+1, \cdots ,M_n^{k+l}\} 
	\]
	such that 
	\begin{enumerate}[(i)]
		\item
		$\pi^{n}_{w_{k+1}\cdots w_l}(v_{\lambda_k^{k+l}(i)}^i)(e_{0}\otimes e_{0}\otimes\cdot\cdot\cdot\otimes e_{0})=C(e_{0}\otimes e_{0}\otimes\cdot\cdot\cdot\otimes e_{0}),$  for $i\in  \{L_n^k, \cdots, M_n^k\},$
		\item
		$\pi^{n}_{w_{k+1}\cdots w_l}(v_{\lambda_k^{k+l}(i)}^j)(e_{0}\otimes e_{0}\otimes\cdot\cdot\cdot\otimes e_{0})=0, ~~$ for $j\in  \{L_n^k,L_n^k+1, \cdots, M_n^k\}\setminus\{i\},$
		\item
		$\pi^{n}_{w_{k+1}\cdots w_l}((v_{\lambda_k^{k+l}(i)}^i)^*)(e_{0}\otimes e_{0}\otimes\cdot\cdot\cdot\otimes e_{0})=C(e_{0}\otimes\cdot\cdot\cdot\otimes e_{0}),$  for $i\in  \{L_n^k, \cdots, M_n^k\},$
		\item
		$\pi^{n}_{w_{k+1}\cdots w_l}((v_{\lambda_k^{k+l}(i)}^j)^*)(e_{0}\otimes e_{0}\otimes\cdot\cdot\cdot\otimes e_{0})=0~~$ for $j\in  \{L_n^k,L_n^k+1, \cdots ,M_n^k\}\setminus\{i\}$.
	\end{enumerate}
	We denote it by $D_{w_1w_2\cdots w_k}\hookrightarrow D_{w_1w_2\cdots w_{k+l}}$. The map 
	$\lambda_k^{k+l}$ is called the embedding map for $D_{w_1w_2\cdots w_k}\hookrightarrow D_{w_1w_2\cdots w_{k+l}}$. 
	\edfn
	\brmrk One can define the above notion for the diagram of irreducible representations of $\mathcal{O}(G_q)$ by replacing a proper analogue of $L_n^k$ and $M_n^k$.  We will simply say that $G_q$ has diagram embeddable property if $D_{w_1w_2\cdots w_k}\hookrightarrow D_{w_1w_2\cdots w_{k+l}}$ for any Weyl word $w$.
	\ermrk

	\bppsn \label{Diagram embedding} Let $1 \leq k,l \leq n$ such that $k+l \leq n$. 
	Assume that 
	$$D_{w_1w_2\cdots w_{k}} \hookrightarrow D_{w_1w_2\cdots w_{k+1}}\hookrightarrow \cdots D_{w_1w_2\cdots w_{k+l}}.$$
	Then  one has 
	$$D_{w_1w_2\cdots w_k}\hookrightarrow D_{w_1w_2\cdots w_{k+l}}.$$
	\eppsn 
	\prf  It is enough to show the claim for $l=2$. For higher values of $l$, the same procedure will work.  Let $\lambda_k^{k+1}$ and $\lambda_{k+1}^{k+2}$ be the embedding  maps for $D_{w_1\cdots w_{k}} \hookrightarrow D_{w_1\cdots w_{k+1}}$ and $D_{w_1\cdots w_{k+1}} \hookrightarrow D_{w_1w_2\cdots w_{k+2}}$, respectively. We will show that  $ \lambda_{k+1}^{k+2} \circ \lambda_k^{k+1}$ is an embedding map for $D_{w_1w_2\cdots w_{k}} \hookrightarrow D_{w_1w_2\cdots w_{k+2}}$.  For that,  fix $ i \in \{ L_n^k,L_n^{k}+1, \cdots, M_n^k\}$. Observe from the diagram representation that for $ j \notin \{ L_n^{k+1},L_n^{k+1}+1, \cdots, M_n^{k+1}\}$, one has 
	$$\pi^{n}_{w_{k+1}}(v_{j}^i)=0.$$ 
	Using the facts  that $D_{w_1w_2\cdots w_{k}} \hookrightarrow D_{w_1w_2\cdots w_{k+1}}$ and $D_{w_1w_2\cdots w_{k+1}}\hookrightarrow  D_{w_1w_2\cdots w_{k+2}}$,  we get
	\begin{IEEEeqnarray*}{rCl}
		\pi^{n}_{w_{k+1} w_{k+2}}(v_{\lambda_{k+1}^{k+2} \circ \lambda_k^{k+l}(i)}^i)(e_{0}\otimes \cdot\cdot\cdot\otimes e_{0})&=& 	(\pi^{n}_{w_{k+1}} \otimes \pi^{n}_{w_{k+2}}) \Delta(v_{\lambda_{k+1}^{k+2} \circ \lambda_k^{k+l}(i)}^i)(e_{0}\otimes \cdot\cdot\cdot\otimes e_{0})\\
		&=&\sum_{j=1}^{2n+1}\pi^{n}_{w_{k+1}}(v_{j}^i) \otimes  \pi^{n}_{w_{k+2}}(v_{\lambda_{k+1}^{k+2} \circ \lambda_k^{k+l}(i)}^j) (e_{0}\otimes \cdots \otimes e_{0})\\
		&=&\sum_{j=L_n^{k+1}}^{M_n^{k+1}}\pi^{n}_{w_{k+1}}(v_{j}^i) \otimes  \pi^{n}_{w_{k+2}}(v_{\lambda_{k+1}^{k+2} \circ \lambda_k^{k+l}(i)}^j) (e_{0}\otimes \cdots \otimes e_{0})\\
		&=& \pi^{n}_{w_{k+1}}(v_{\lambda_{k}^{k+1}(i) }^i)  \otimes  \pi^{n}_{w_{k+2}}(v_{\lambda_{k+1}^{k+2} \circ \lambda_k^{k+l}(i)}^{\lambda_{k+1}^{k+2}(i)}) (e_{0}\otimes \cdots \otimes e_{0})\\
		&=& C (e_{0}\otimes \cdots \otimes e_{0}).
	\end{IEEEeqnarray*}
	The other conditions can be verified in a similar way. 
	\qed 
	\brmrk 
	Diagram embeddability is implicitly but crucially used in (\cite{ChaSau-2018aa}, \cite{ChaSau-2019aa}), where the authors proved this property for the $A$, $C$, and $D$ type diagrams.  Here we  formulate this notion and prove it for type $B$ diagrams. 
	\ermrk
	
	\section{GKdim of $\mathcal{O}(SO_q(2n+1))$-modules }
	
	In this section, we compute  GKdim of a simple unitarizable $\mathcal{O}(SO_q(2n+1))$-module $V_{t,w}$ for  $t \in \bbbt^n$ and $w \in W_n$. 
	
	\bdfn \label{3.1.}   
	Let $A$ be a finitely generated unital associative algebra, and let $N$ be a finitely generated left $A$-module.  Let $\Xi$ be the set of all finite subsets of $A$ that contains $1$, and generates the algebra $A$.  Let $\Sigma$ be the set of all finite-dimensional subspaces of $N$ that generates $N$ as a left $A$-module. The Gelfand-Kirillov  dimension of $N$ is defined as follows:
	\[GKdim(N) = \sup_{\xi\in \Xi,\, W \in \Sigma} \limsup_{r\to\infty} \frac{\ln(\dim(\xi^{r}W))}{\ln\,r}.\]
	In fact,  if $\xi \in \Xi$ and $W \in \Sigma$, then one can show that 
	\[
	GKdim(N) = \limsup_{r\to \infty} \frac{\ln(\dim(\xi^{r}W))}{\ln\,r}. 
	\]
	\edfn
	
	Throughout this section, our choice for  $\xi$ will be $\{v^{k}_{l}: 1\leq k,l\leq 2n+1\}\cup \{1\}$, and for $W$, it will be  the one dimensional subspace spanned by the vector $e_0\otimes e_0\otimes \cdots \otimes e_0$.   In the following proposition, we obtain an  upper bound of the GKdim of a simple unitarizable left $\mathcal{O}(SO_{q}(2n+1))$-module.
	
	\begin{ppsn}\label{3.1}
		For $w\in W_{n}$ and $t\in \mathbb{T}^{n}$, let $V_{t,w}$ be the associated simple unitarizable left $\mathcal{O}(SO_{q}(2n+1))$-module.   Then we have,
		$$GKdim(V_{t,w})\leq l(w).$$
	\end{ppsn}
	\begin{proof}
		Recall first from (Proposition 4, page 100,\cite{KliSch-1997aa}) that 
		$$
		\{\alpha^a(\beta^*)^b\beta^c,\, (\alpha^*)^d(\beta^*)^e\beta^f \mid a,b,c,d,e,f \in \mathbb{N} \}$$
		is a linear basis of the algebra $\mathcal{O}(SU_q(2))$. 
		Now, for the monomials $(v_{1,\mathfrak{sl}_2}^1)^a(v_{1,\mathfrak{sl}_2}^2)^b(v_{2,\mathfrak{sl}_2}^1)^c$ and $(v_{2,\mathfrak{sl}_2}^2)^d(v_{1,\mathfrak{sl}_2}^2)^e(v_{2,\mathfrak{sl}_2}^1)^f$,  define their exponents to be $a$ and $d$, respectively.   The exponent of a linear combination of these basis elements is defined to be the maximum of the exponents of monomials appearing in the linear combination. Define
		\[
		M=\max_{\{v\in \xi,\,i\in \{1,2,\cdots , n\}\}}\{\text{exponent of~} \phi_{i}^* (v):v\in\xi \}.
		\]
		Using Theorem \ref{homomorphism}, one can see that the linear span of $\{\pi_{t,w}^n(\xi^r)(e_0\otimes \cdots\otimes e_0)\}$ is contained in the span of $\{e_{i_1}\otimes\cdots\otimes e_{i_{l(w)}} :0\leq  i_j\leq rM, 1\leq j \leq n\}$.
		Using this, we get
		\[
		\text{GKdim}(V_{t,w})=\limsup_{r \rightarrow \infty} \frac{\ln \text{dim}(\xi^{r}(e_0\otimes e_0\cdots\otimes e_0))}{\ln Mr}\leq \limsup_{r \rightarrow \infty} \frac{\ln (Mr)^{\ell(w)}}{\ln r}=\ell(w).
		\]
		
	\end{proof}
	\blmma \label{4.2a}
	Let  $w$ be a Weyl word with  $\ell(w_n)=r$.    Then there exist polynomials $P_{1}^{(w,n)},P_{2}^{(w,n)},\\\cdots ,P_{r}^{(w,n)}$ with non-commutating variables $\pi_{w}^n(v^{2n+1}_{j})$'s and a permutation $\sigma$ of $\{1,2,...,r\}$ such that for all $v\in c_{00}(\mathbb{N})^{\otimes \sum_{i=1}^{n-1}\ell(w_i)}$ and $z_1^n,z_2^n,...,z_{r}^n\in\mathbb{N},$ one has
	$$(P_1^{(w,n)})^{z_{\sigma(1)}^{n}}(P_2^{(w,n)})^{z_{\sigma(2)}^{n}}\cdots (P_1^{(w,n)})^{z_{\sigma(r)}^{n}}(v\otimes e_{0}\otimes e_{0}\otimes\cdot\cdot\cdot\otimes e_{0})=C\,v\otimes e_{z_{1}^n}\otimes e_{z_{2}^n}\otimes\cdot\cdot\cdot\otimes e_{z_r^n},$$
	where $C$ is a nonzero real number.
	\elmma 
	\begin{proof}
		We divide the proof into two cases. \\
		\textbf{ Case 1}:  $r> n$. In this case, we have  $w=w_1w_2\cdot\cdot\cdot w_{n-1}s_1s_2\cdot\cdot\cdot s_n s_{n-1}\cdot\cdot\cdot s_{2n-r}$.  Define  
		\[
		P_j^{(w,n)}=\begin{cases} \pi_w^n(v_{2n-j+2}^{2n+1}) & \mbox{ if } 1\leq j\leq 2n-r-1,\cr 
			\pi_w^n(v_{j+1}^{2n+1}) & \mbox{ otherwise}.\cr 
		\end{cases} 
		\]
		The permutation $\sigma$ is defined as follows:
		\[
		\sigma(j)=\begin{cases} 2n-j & \mbox{ if } 2n-r\leq j\leq r,\cr
			j & \mbox{ otherwise}.\cr
		\end{cases}
		\]
		Using the diagram  associated to the representation $\pi_w^n$, we get
		\[
		P_j^{(w,n)}\thicksim\begin{cases} 1^{\otimes\sum_{i=1}^{n-1}\ell(w_i)}\otimes 1^{\otimes\sigma(j)-1}\otimes \underbrace{\sqrt{1-q^{4N}}S^*}_{\sigma(j)-th~~ \text{place}}\otimes 1^{r-\sigma(j)} & \mbox{ if } j\neq n,\cr
			1^{\otimes\sum_{i=1}^{n-1}\ell(w_i)}\otimes 1^{\otimes\sigma(n)-1}\otimes\underbrace{S^{*}\sqrt{(1+q^2)(1-q^{2N+2})}q^{N}}_{\sigma(n)-th~~ \text{place}}\otimes 1^{r-\sigma(n)}        & \mbox{ if } j=n.\cr
		\end{cases}
		\]
		on the subspace generated by basis element $e_0$ at $\sigma(k)$-th place for $k<j$. Using the identity $S^{*}\sqrt{(1+q^2)(1-q^{2N+2})}q^{N}=\sqrt{(1+q^2)(1-q^{2N})}q^{N-1}S^*$, we have
		\[
		(P_j^{(w,n)})^z\thicksim\begin{cases} 1^{\otimes\sum_{i=1}^{n-1}\ell(w_i)}\otimes 1^{\otimes\sigma(j)-1}\otimes \underbrace{(\sqrt{1-q^{4N}}S^*)^z}_{\sigma(j)-th~~ \text{place}}\otimes 1^{r-\sigma(j)} & \mbox{ if } j\neq n,\cr
			1^{\otimes\sum_{i=1}^{n-1}\ell(w_i)}\otimes 1^{\otimes\sigma(n)-1}\otimes\underbrace{(\sqrt{(1+q^2)(1-q^{2N})}q^{N-1}S^{*})^z}_{\sigma(n)-th~~ \text{place}}\otimes 1^{r-\sigma(n)}        & \mbox{ if } j=n.\cr
		\end{cases}
		\]
		With these expressions in hand, one can directly verify the claim.\\
		\noindent \textbf{ Case 2: $r \leq n$.} In this case, we have $w=w_1w_2\cdot\cdot\cdot w_{n-1}s_1s_2\cdot\cdot\cdot s_{r}$. Define $\sigma(j):=j$ for all $\{1,2,\cdots ,r\}$. Let  $P_j^{(w,n)}:=\pi^{n}_{w}(v^{2n+1}_{2n-j+2}).$ Following steps given in the previous case, we get the claim.
	\end{proof}
	To extend the above result to all parts of a Weyl word, we will use the diagram embedding defined in the previous section. 
	\begin{ppsn}\label{3.3}
		Let $w \in W_n$. Then one has the following:
		$$D_{w_1} \hookrightarrow D_{w_1w_2}\hookrightarrow \cdots D_{w_1w_2\cdots w_n}.$$
	\end{ppsn}
	\begin{proof} 
		Fix $1 \leq i \leq n$. 	From Proposition \ref{3.2}, we see that the Weyl word $w_{i+1}$ is either of the form $s_{i}s_{i+1}\cdots s_{k}$      
		or $s_{i}s_{i-1}\cdots s_{n-1}s_{n}s_{n-1}\cdots s_{k}.$
		Therefore for all $1\leq k<n$, $s_k$ occurs in $w_{i+1}$ either once or twice or does not occur. \\For $j\leq n$, we define
		\[
		\lambda_i^{i+1}(j)=\begin{cases} j-1 & \mbox{ if } s_{j-1} ~\text{occurs in}~ w_{i+1}~~\text{once},\cr
			j & \mbox{ otherwise}.\cr
		\end{cases}
		\]
		For $j> n+1$, we define
		\[
		\lambda_i^{i+1}(j)=\begin{cases} j+1 & \mbox{ if } s_{2n-j+1} ~\text{occurs in}~ w_{i+1}~~\text{once},\cr
			j & \mbox{ otherwise}.\cr
		\end{cases}
		\]
		For $j=n+1$, we define $\lambda_i^{i+1}(n+1)=n+1$. We can now directly verify all the conditions for diagram embedding  $D_{w_1w_2\cdots w_{i}} \hookrightarrow  D_{w_1w_2\cdots w_{i+1}}$ given in (\ref{definition of diagram embedding}) using diagram representation.
	\end{proof}
	For $w \in W_n$, consider the module $V_w$ and  a generic vector $v=v_{<i} \otimes v_i \otimes v_{>i}\in V_w$, where $ v_{<i} \in V_{w_1} \otimes \cdots \otimes V_{w_{i-1}}$, $v_i\in V_{w_i}$, and $ v_{>i} \in V_{w_{i+1}} \otimes \cdots\otimes V_{w_n}$. In the following lemma,  we  produce, using the diagram embeddings proved above,  some endomorphisms  for each $1\leq i \leq n$ that acts only the $i$-th part $v_i$ of a vector $v$, keeping  other parts the same provided $v_{>i}=e_0\otimes \cdots\otimes e_0$.  More importantly, these endomorphisms acts on $v_i$ in the same way as the variables mentioned in Lemma \ref{4.2a} assuming that the rank of the Lie algebra is $i$. Since we are talking about $\mathfrak{so}(2n+1)$ and $\mathfrak{so}(2i+1)$ at the same time, it is important to distinguish the generators of $ \mathcal{O}(SO_q(2i+1))$, and $ \mathcal{O}(SO_q(2n+1))$.  We denote by $(v_{l,s}^{k})$ the standard generators of $ \mathcal{O}(SO_q(2s+1))$.
	\blmma \label{lowerrank}
	Let $V_w$ be the $ \mathcal{O}(SO_q(2n+1))-$module associated to the Weyl word $w$. Fix $1\leq i \leq n$.  Then there exist  endomorphisms $T_1^i,T_2^i,\cdots ,T_{2i+1}^i\in \{\pi_{w}^{n}(v_{l,n}^k):1\leq l,k\leq 2n+1\}$ such that for all $1\leq j\leq 2i+1$, we have
	\begin{align}
		T^{i}_{j}(u_1\otimes u_2 \otimes e_0\otimes\cdots \otimes e_0)
		&=C\,u_1\otimes(\pi_{w_i}^i(v_{j,i}^{2i+1})u_2)\otimes e_0\otimes\cdots\otimes e_0,
	\end{align}
	where  $u_{1}\in V_{w_1}\otimes V_{w_2}\otimes\cdots \otimes V_{w_{i-1}}$ and $u_2\in V_{w_{i}}$.
	\elmma
	\begin{proof}
		From Proposition \ref{3.3}, we have 
		$$D_{w_1w_2\cdots w_i} \hookrightarrow D_{w_1w_2\cdots w_n}.$$ Let $\lambda_{i}^n$ be the associated embedding map.    For $1 \leq j \leq 2i+1$, define 
		$$T_j^i=\pi_{w}^n(v_{\lambda_{i}^n(j),n}^{n+i+1}).$$
		One can see, using the diagram representation that  if $k \notin \{n-i+1,\cdots ,n+i+1\}$, then $$\pi_{w_1\cdots w_i}^{n}(v_k^{n+i+1})=0.$$
		Using this and the definition of diagram embedding (\ref{definition of diagram embedding}), we get
		\begin{IEEEeqnarray*}{rCl}
			T^{i}_{j}(u_1\otimes u_2 \otimes e_0\otimes\cdots \otimes e_0)
			&=&\pi_{w}^n(v_{\lambda_{i}^n(j),n}^{n+i+1})(u_1\otimes u_2 \otimes e_0\otimes\cdots \otimes e_0)\\
			&=&\sum_{k=1}^n\pi_{w_1\cdots w_i}^n(v_{k,n}^{n+i+1})(u_1\otimes u_2)\otimes \pi_{w_{i+1}\cdots w_n}^n(v_{\lambda_{i}^n(j),n}^{k})(e_0\otimes\cdots \otimes e_0)\\
			&=&\sum_{k=L_n^i}^{M_n^i}\pi_{w_1\cdots w_i}^n(v_{k,n}^{n+i+1})(u_1\otimes u_2)\otimes \pi_{w_{i+1}\cdots w_n}^n(v_{\lambda_{i}^n(j),n}^{k})(e_0\otimes\cdots \otimes e_0)\\
			&=&C\,\pi_{w_1w_2\cdots w_i}^i(v_{j,i}^{2i+1})(u_1\otimes u_2)\otimes e_0\otimes\cdots \otimes e_0\\
			&=&C \,u_1\otimes \pi_{w_1w_2\cdots w_i}^i(v_{j,i}^{2i+1})(u_2)\otimes e_0\otimes\cdots \otimes e_0.
		\end{IEEEeqnarray*}
	\end{proof}
	\begin{lmma}\label{3.5}
		Suppose $w\in W_n$. Then for each $1\leq i \leq n$ and $1\leq j\leq \ell(w)$, there exist polynomials $P_j^{(w,i)}$ with noncommuting variables $\pi_w^n(v_l^k)$'s and a permutation $\sigma_i$ of $\{1,2,\cdots,\ell(w_i)\}$ such that
		\begin{align*}
			&(P_1^{(w,n)})^{r^n_{\sigma_n(1)}}(P_2^{(w,n)})^{r^n_{\sigma_n(2)}}\cdots  (P_{\ell(w_n)}^{(w,n)})^{r^n_{\sigma_n(\ell(w_n))}}\cdots
			\cdots(P_1^{(w,1)})^{r^n_{\sigma_n(1)}}(P_2^{(w,1)})^{r^n_{\sigma_n(2)}}\cdots \\
			&(P_{\ell(w_1)}^{(w,1)})^{r^n_{\sigma_n(\ell(w_1)}}(e_0\otimes e_0\cdots\otimes e_0) 
			=C\, e_{r_1^1}\otimes e_{r_2^1}\cdots\otimes e_{r_{l(w_1)}^1}\cdots \otimes e_{r_{1}^n}\otimes e_{r_2^n}\otimes\cdots\otimes e_{r^n_{\ell(w_n)}},
		\end{align*}
		where $C$ is a nonzero constant.
	\end{lmma}
	\begin{proof} 
		Take the same polynomials and the same permutation as given in Lemma $\ref{4.2a}$ to define $P_j^{(w,n)}$,  $1\leq j \leq \ell(w_n)$ and the permutation $\sigma_n$. Let $1\leq i <n$. We view $w_1w_2\cdots w_i$ as an element of the Weyl group $W_i$. Hence we can define polynomials $P_j^{(w_1\cdots w_i,i)}$ for  $1\leq j \leq \ell(w_i)$ and permutations $\sigma_i$ using Lemma $\ref{4.2a}$. To define  $P_j^{(w,i)}$,  replace the variables $\pi_{w_i}^i(v_{l,i}^{2i+1})$ in  $P_j^{(w_1\cdots w_i,i)}$ with $T_l^{n+i+1}$, $1\leq l \leq 2i+1$. The permutation $\sigma_i$ remains the same. The claim now follows by a direct verification using Lemma \ref{3.3} and Lemma \ref{lowerrank}.
	\end{proof}
	
	\begin{thm}\label{GKdim1}
		Let $t \in \bbbt^n$ and  $w\in W_n$. Then we have $$\text{GKdim}(V_{t,w})=\ell(w).$$
	\end{thm}
	\begin{proof}
		In view of Proposition \ref{3.1}, it is enough to show that $$\text{GKdim}(V_{t,w})\geq\ell(w).$$ Since  GKdim($V_{w,t}$)=GKdim($V_w$), it suffices to prove the claim
		for $V_{w}$. 
		Define 
		$$A:=max\{\text{ degree of } P_{j}^{(w,n)}:1\leq i\leq n,1\leq j \leq \ell(w_i)\}$$
		From Lemma $\ref{3.5}$, it follows that
		$$\text{span} ~\pi_w^n(\xi^{Ar})(e_0\otimes e_0\otimes\cdots\otimes e_0)\supset \text{span}\{{e_{\beta_1}\otimes e_{\beta_2}\otimes\cdots\otimes e_{\beta_{\ell(w)}}:\sum_{i=1}^{\ell(w)}\beta_{i}=r}\}.$$
		So we get,
		$$\text{dim (span~} \pi_w^n(\xi^{Ar})(e_0\otimes e_0\otimes\cdots\otimes e_0))\geq {r+\ell(w)-1  \choose r }.$$
		Therefore, we have
		$$GKdim(V_w)\geq \limsup \frac{\text{ln}~ \mbox{dim}(\text{ span}~ \pi_w^n(\xi^{Ar})(e_0\otimes e_0\otimes\cdots\otimes e_0))}{\text{ln}\, r}\geq \limsup \frac {\text{ln}{r+\ell(w)-1  \choose r} }{\text{ln} ~r}=\ell(w).$$\\
	\end{proof}
	To compute GKdim of $V_{t,w}^{\mbox{Spin}}$, our choice for the set of finite generators of the algebra $\mathcal{O}(\mbox{Spin}_q(2n+1))$ will be  $\Gamma=\xi \cup \zeta$, where $\zeta:=\{x_j^i:1\leq i,j\leq 2^n\}$ consists of the generators of $\mathcal{O}(\mbox{Spin}_q(2n+1))$, and the choice for the finite dimensional vector space will be the subspace generated by $e_0\otimes e_0 \otimes \cdots \otimes e_0$. 
	\bthm \label{GKdim2}
	Let $t \in \bbbt^n$ and  $w\in W_n$. Then we have $$\text{GKdim}(V_{t,w}^{\mbox{Spin}})=\ell(w).$$
	\ethm 
	\prf 
	First observe that if we replace $\xi$ by $\Gamma=\xi \cup \zeta$ in the proof of Proposition \ref{3.1}, we get 
	\[
	\text{GKdim}(V_{t,w}^{\mbox{Spin}})\leq \ell(w).
	\] 
	From the remark $(\ref{restriction})$, we get 
	\[
	\text{dim (span~} \pi_w^n(\xi^{r})(e_0\otimes e_0\otimes\cdots\otimes e_0))\leq \text{dim (span~} \pi_w^n(\Gamma^{r})(e_0\otimes e_0\otimes\cdots\otimes e_0))
	\]
	for any $r \in \bbn$. The result follows from this  and Theorem \ref{GKdim1}.
	\qed 
\section{GKdim of  homogeneous spaces of $SO_q(2n+1)$}
Given a subset   $R$  of $\Pi$, one can associate a quotient space of $\mbox{Spin}_q(2n+1)$ and $SO_q(2n+1)$. In this section, our main objective is to show that  GK dim of these quotient spaces is same as the manifold dimension of their classical counterpart.  Let us first recall from \cite{NesTus-2012aa} the definition of the quantized function algebra on a quotient space of $\mbox{Spin}_q(2n+1)$ and $SO_q(2n+1)$.

Fix a subset $R$ of $\Pi = \{\alpha_1, \cdots, \alpha_n\}$.  Let $K^R$ and $\tilde{K}^R$ be  closed connected subgroups of $\mbox{Spin}(2n+1)$ and $SO(2n+1)$ respectively,  having the same complexified Lie algebra ${\mathfrak{k}}^R$, which is    generated by the elements $E_i$, $F_i$ and $H_i$ corresponding to $\alpha_i\in R$. 
Let  $P(R^c)$ subgroup of the weight lattice $P$, generated by the fundamental weight $w_i$ for $i\in R^c=\Pi-R$. Also, we identify $P$ with  $\mathbb{T}^n$, and we denote the annihilator of $P(R^c)$ in $\mathbb{T}^n$ by $\mathbb{T}^n_{P(R^c)}$. 

\begin{dfn}
	Let $R \subset \Pi$ and  $I_R=\{i \in \{1,2,\cdots , n\}: \alpha_i \in R\}$. The quantized algebra $\mathcal{O}(\mbox{Spin}_q(2n+1)/K_q^{R,P(R^c)})$ is defined as   the $\ast$-subalgebra of $\mathcal{O}(\mbox{Spin}_q(2n+1))$ consisting of all element $v$ which satisfies the following:
	\[
	E_iv = F_iv = 0, \, \,
	tv = v, \,\, \mbox{ for all } t \in \mathbb{T}^n_{P(R^c)}, \, i  \in I_R.
	\]
	Here the action of $E_i$'s, $F_i$'s and $t \in \mathbb{T}^n_{P(R^c)}$ is  as mentioned in  equation (\ref{module}) and equation  (\ref{torus}).\\
\noindent The quantized algebra $\mathcal{O}(SO_q(2n+1)/K_q^{R,P(R^c)})$  of regular functions  on the quotient space   $SO_q(2n+1)/K_q^{R,P(R^c)}$ is defined as follows:
	\[
\mathcal{O}(SO_q(2n+1)/K_q^{S,P(R^c)}) = \{v \in \mathcal{O}(SO_q(2n+1)): E_iv = F_iv = 0, 
 tv = v ~\forall i \in I_R, t \in \mathbb{T}^n_{P(R^c)}\}.
	\]
	It is a   $\ast$-subalgebra of  $\mathcal{O}(SO_q(2n+1))$ as well as $\mathcal{O}(\mbox{Spin}_q(2n+1)/K_q^{R,P(R^c)})$.
	\edfn
The representation $\pi^n_{t,w}$ when restricted to $\mathcal{O}(SO_q(2n+1)/\tilde{K}_q^{R,P(R^c)})$ gives a irreducible representation of $\mathcal{O}(SO_q(2n+1)/\tilde{K}_q^{R,P(R^c)})$. In the same way, the irreducible representation $\pi^{n, \mbox{Spin}}_{t,w}$ of  $\mathcal{O}(\mbox{Spin}_q(2n+1))$ induces an irreducible representation of $\mathcal{O}(\mbox{Spin}_q(2n+1)/K_q^{R,P(R^c)})$. We  continue to denote these new representations and the corresponding modules by the same notations.  Observe that $\pi^{n}_{t,w}$ is the restriction of  $\pi^{n, \mbox{Spin}}_{t,w}$ to $\mathcal{O}(SO_q(2n+1)/\tilde{K}_q^{R,P(R^c)})$.
\begin{thm}{\cite{NesTus-2012aa}}\label{4.2}
	For  $w\in W_n^R$ and $t\in\mathbb{T}^n$,  the representation $\pi^{n, \mbox{Spin}}_{t,w}$  of $\mathcal{O}(\mbox{Spin}_q(2n+1)/K_q^{R,P(R^c)})$  is irreducible. All irreducible representations of $\mathcal{O}(\mbox{Spin}_q(2n+1)/K_q^{R,P(R^c)})$ is of the above type up to equivalence.  Moreover, two representations, $\pi^{n, \mbox{Spin}}_{t_1,w_1}$  and $\pi^{n, \mbox{Spin}}_{t_2,w_2}$ , are equivalent if and only if $w_1=w_2$ and $t_2t_1^{-1}\in\mathbb{T}^n_{P(R^c)}$. In this case, these representations are actually equal. 
\end{thm}
Let $\mathcal{P}(\scrt)$ be the subalgebra of $\text{End}(c_{00}(\mathbb{N}))$ generated by $\sqrt{1-q^{2N+2}}S$, $S^*\sqrt{1-q^{2N+2}}$, and $q^N$. Let $\mathcal{P}(C(\mathbb{T}))$ be the subalgebra  of $\text{End}(c_{00}(\mathbb{Z}))$  generated by shift operator $S$ and its adjoint $S^{*}$.  For $w \in W_n^R$, let 
$$\eta_{w}^{n,\mbox{Spin}}:\mathcal{O}
(\mbox{Spin}_q(2n+1)/K_q^{R,P(R^c)})\longrightarrow\mathcal{P}
(C(\mathbb{T}^n/\mathbb{T}^n_{P(R^c)}))\otimes \mathcal{P}(\scrt)^{\otimes \ell(w)}$$ 
be a homomorphism given by:
\begin{align*}
	\eta_{w}^{n,\mbox{Spin}}(a)([t])=\pi_{t,w}^n(a), \, \mbox{ for all } a\in \mathcal{O}(\mbox{Spin}_q(2n+1)/K_q^{R,P(R^c)}),\, t\in\mathbb{T}^n/\mathbb{T}^n_{P(R^c)}.
\end{align*}
Theorem $\ref{4.2}$ establishes  that  the homomorphism 	$\eta_{w}^{n,\mbox{Spin}}$ is well-defined.   The following theorem provides us a concrete faithful realization of the algebra $\mathcal{O}(\mbox{Spin}_q(2n+1)/K_q^{R,P(R^c)})$, which we will use later to compute the invariant.
\begin{thm}\label{4.3}
Let   $w_n^R$ denote the longest word of  $W_n^R$. Then  homomorphism
	$$
	\eta_{w_n^R}^{n,\mbox{Spin}}:\mathcal{O}
	(\mbox{Spin}_q(2n+1)/K_q^{R,P(R^c)})\longrightarrow\mathcal{P}(C(\mathbb{T}^n/\mathbb{T}^n_{P(R^c)}))\otimes \mathcal{P}(\scrt)^{\otimes \ell(w_n^R)}.$$
	is faithful. 
	\end{thm}
	We omit the proof as it is an immediate consequence of Theorem 5.9 of $\cite{StoMat-1999aa}$. To see the details, we refer the reader to Theorem $2.5$ of \cite{ChaSau-2019aa}.
	Viewing $\mathcal{O}(SO_q(2n+1)/\tilde{K}_q^{R,P(R^c)})$ as a $\ast$-subalgebra of $\mathcal{O}(\mbox{Spin}_q(2n+1)/K_q^{R,P(R^c)})$, one can restrict the homomorphism 	$\eta_{w_n^R}^{n,\mbox{Spin}}$ to  $\mathcal{O}(SO_q(2n+1)/\tilde{K}_q^{R,P(R^c)})$ and get a homomorphism $\eta_{w_n^R}^n$. 
\bcrlre
Let   $w_n^R$ denote the longest word of  $W_n^R$. Then  homomorphism
$$
\eta_{w_n^R}^{n}:\mathcal{O}(SO_q(2n+1)/\tilde{K}_q^{R,P(R^c)})\longrightarrow\mathcal{P}(C(\mathbb{T}^n/\mathbb{T}^n_{P(R^c)}))\otimes \mathcal{P}(\scrt)^{\otimes \ell(w_n^R)}.$$
is faithful. 
\ecrlre
\prf It follows from the fact that $\eta_{w_n^R}^{n}=	\restr{\eta_{w_n^R}^{n,\mbox{Spin}}}{\mathcal{O}(SO_q(2n+1)/\tilde{K}_q^{R,P(R^c)})}$ and Theorem \ref{4.3}.
\qed

\bdfn For a  unital  algebra $A$, the Gelfand-Kirillov dimension  is defined as follows:
\[
\mbox{GKdim}(A) = \sup_{\xi} \limsup_{r\to\infty} \frac{\ln\, \dim(\xi^{r})}{\ln\,r}.
\]
Here $\xi$ varies over all finite sets of $A$ that  contains $1$. 
\edfn
Some properties of the GKdim are mentioned below.
\begin{itemize}
	\item If $B$ is a subalgebra of $A$ containing the multiplicative identity of $A$, then GKdim$(B) \leq$ GKdim$(A)$, (see \cite{Row-1991aa}).
	\item GKdim$(A \otimes B)\leq$  GKdim$(A) +$ GKdim$(B)$.
	\item GKdim($\mathcal{P}(C(\mathbb{T}))=1$ and GKdim$(\mathcal{P}(\scrt))=2$ (See Proposition $3.3$, \cite{ChaSau-2019aa}).
\end{itemize}
Invoking these results,  we obtain the following lower bound for  GKdim$(\mathcal{O}(\mbox{Spin}_q(2n+1)/K_q^{R,P(R^c)}))$.
\begin{ppsn}\label{4.5}
	Let $w_n^R $ be the longest element of $W_n^R$. Then we get the following:
	$$\text{GKdim}(\mathcal{O}(\mbox{Spin}_q(2n+1)/K_q^{R,P(R^c)}))\leq 2\ell(w_n^R)+ \# P(R^c).$$
\end{ppsn}
\begin{proof}
	The algebra $\mathcal{O}(\mbox{Spin}_q(2n+1)/K_q^{R,P(R^c)})$ can be view as a subalgebra of $\mathcal{P}(C(\mathbb{T}^n/\mathbb{T}^n_{P(R^c)}))\otimes \mathcal{P}(\mathcal{J})^{\otimes \ell(w_n^R)}$, thanks to Theorem $\ref{4.3}$. Thus, we have following inequality:
	$$\text{GKdim}(\mathcal{O}(\mbox{Spin}_q(2n+1)/K_q^{R,P(R^c)}))\leq \text{GKdim}(\mathcal{P}(C(\mathbb{T}^n/\mathbb{T}^n_{P(R^c)}))\otimes \mathcal{P}(\scrt)^{\otimes \ell(w_n^R)})\leq 2\ell(w_n^R)+ \# P(R^c).$$
\end{proof}

We will show that equality holds in Proposition $\ref{4.5}$ for specific choices of $R$. Let $R_m$ be a subset of $\Pi$ defined as follows:

\[
R_{m}=\begin{cases} \emptyset & \mbox{ for } m=1,\cr 
	\{n-m+2,\cdots,n\} & \mbox{ for } 2\leq m\leq n.\cr \end{cases}
\]
To get  an appropriate finite subset  of $\mathcal{O}(\mbox{Spin}_q(2n+1)/K_q^{R_m,P(R_m^c)})$,
define $$\zeta_m =\{v_l^k:1\leq l \leq 2n+1, k \in \{1,\cdots,n-m+1\}\cup \{n+m,\cdots,2n+1\} \}\cup \{1\}.$$
It is easy to verify that $$\zeta_m\subset \mathcal{O}(SO_q(2n+1)/K_q^{R_m,P(R_m^c)})\subset\mathcal{O} (\mbox{Spin}_q(2n+1)/K_q^{R_m,P(R_m^c)}).$$  
Let $\phi$ be an isomorphism between $\mathbb{T}^{n-m+1}$ and $\mathbb{T}^n/\mathbb{T}^n_{P(R^c_m)}$ given by; 
$$\phi(t)=[(t,1,\cdots,1)]~~\forall\, t\in \mathbb{T}^{n-m+1}.$$
Consider the action $\eta_{e}^n$ of $\mathcal{O}(SO_q(2n+1)/K_q^{R_m,P(R_m^c)})$ on the vector space $c_{00}(\mathbb{Z})^{\otimes (n-m+1)}$. So we can write as $\eta_e^n=\tau_{[t,1,\cdots,1]}(a)$ for all $a\in\mathcal{O}(SO_q(2n+1))$ and $t\in\mathbb{T}^{n-m+1}$. For any $w\in W_n^{R_m}$, we have $$\eta_w^n=\eta_e^n\ast\pi_{w}^n.$$
Hence, for $v_l^k\in \zeta_m$, we have
\[
\eta_e^n(v_l^k)=\begin{cases} \delta_{ij} 1^{\otimes 2n-k}\otimes S^{*}\otimes 1^{\otimes k-m-n}  & \mbox{ if } k>n+1,\cr 
	\delta_{ij} 1^{\otimes n-m+1} & \mbox{ if } k=n+1,\cr
	\delta_{ij} 1^{\otimes k-1}\otimes S\otimes 1^{\otimes n-k-m+1}    & \mbox{ if } k< n+1.\cr
\end{cases}
\]
From this, we obtain:
\begin{IEEEeqnarray}{rCl} \label{equation}
	\eta_w^n(v_l^k)=(\eta_{e}^n\otimes \pi_w^n)\bigtriangleup(v^k_l)=\eta_{e}^n(v^k_k)\otimes \pi_w^n(v^k_l).
\end{IEEEeqnarray}
Let us record the following facts, which will be used later in proving our main result. Fix $M \in \bbn$.
\begin{enumerate}[(i)] 
	\item  The set $\{S^i(S^*)^{i+M}: i \in \bbn \}$ consists of  linearly independent endomorphims of $c_{00}(\bbn)$. 
	\item Let $D:e_n \mapsto d_ne_n$  be a diagonal endomorphism. Then $\{(SD)^i(D^*S^*)^{i+M}: i \in \bbn \}$ is a linearly independent  set of endomorphims of $c_{00}(\bbn)$.
\end{enumerate}

\begin{lmma}\label{5.5}
	Let $w\in W_n^{R_m}$. Then there exist a permutation $\sigma$ of $\{1,2,\cdots,\ell(w_n)\}$ and polynomials $h_{0}^{w,n},h_{1}^{w,n},\cdots,h_{\ell(w_n)}^{w,n} ~and~ h_{1*}^{w,n},h_{2*}^{w,n},\cdots,h_{\ell(w_n)*}^{w,n}$ with variables $\eta_{w}^n(v_l^{2n+1})$
	and $\eta_{w}^n((v_l^{2n+1})^*)$ 
	such that
	\begin{enumerate}[(i)]
		\item for all $u\in c_{00}(\mathbb{Z})^{\otimes(n-m)}\otimes c_{00}(\mathbb{N})^{\otimes\sum_{j=m}^{n-1}\ell(w_j)}$, we have
		\begin{IEEEeqnarray*}{lll}
			(h_{0}^{(w,n)})^{r^n_{0}}(h_{1*}^{(w,n)})^{p^n_{\sigma(1)}}(h_{1}^{(w,n)})^{r^n_{\sigma(1)}}
			\cdots (h_{\ell(w_i)^*}^{(w,n)})^{p^n_{\sigma(\ell(w_n))}}(h_{\ell(w_n)}^{(w,n)})^{r^n_{\sigma(\ell(w_n))}}
			(e_0\otimes u\otimes e_0^{\otimes \ell(w_n)}) \\
			=Ce_{r_0^n}\otimes u\otimes  e_{{r_1^n-p_1^n}}\otimes e_{{r_2^n-p_2^n}}\otimes\cdots\otimes  e_{{r_{\ell(w_n)}^n-p_{\ell(w_n)}^n}},
		\end{IEEEeqnarray*}
	where $r_0^n,r_i^n,p_i^n\in\mathbb{N}$ and $r_i^n\geq p_i^n$ for  $1\leq i\leq \ell(w_n).$
		\item The set 
		$$\{(h_{0}^{(w,n)})^{r^n_{0}}(h_{1*}^{(w,n)})^{p^n_{\sigma(1)}}(h_{1}^{(w,n)})^{r^i_{\sigma(1)}}\cdots(h_{\ell(w_n)^*}^{(w,n)})^{p^n_{\sigma(\ell(w_i))}}\\(h_{\ell(w_i)}^{(w,n)})^{r^n_{\sigma(\ell(w_n))}}:~r_j^n\geq p_j^n~r_0^n,r_j^n,p_j^n\in\mathbb{N} \}$$ consists of  linearly independent endomorphisms.
	\end{enumerate}
\end{lmma}
\begin{proof}
	First, define $h_{0}^{w,n}$ to be $\eta_{w}^n(v_{2n+1-\ell(w_n)}^{2n+1})$. It is not difficult to see that endomorphism $\pi_w^n(v_{2n+1-\ell(w_n)}^{2n+1})$ is of the form $1^{\otimes\sum_{i=m}^{n-1}\ell(w_i)}\otimes q^{j_1N}\otimes q^{j_2N}\otimes\cdots\otimes q^{j_{\ell(w_n)}N}$. Hence we have
	$$h_{0}^{w,n}=\eta_{e,n}(v_{2n+1}^{2n+1})\otimes \pi_w^n(v_{2n+1-\ell(w_n)}^{2n+1})\sim S^*\otimes 1^{\otimes (n-k)}\otimes 1^{\otimes \sum_{m=k}^{n-1}\ell(w_m)}. $$ 
	Let $F_j^{(w,n)}$ be the polynomial obtained by replacing the action $\pi_w^n$ in polynomial $P_j^{(w,n)}$ given in Lemma $\ref{4.2a}$ with $\eta_{w}^n$. Define $$h_j^{(w,n)}:=((h_0^{(w,n)})^*)^t F_j^{(w,n)}\, \mbox{ and } \, h_{j*}^{(w,n)}:=(h_j^{(w,n)})^{*},$$ where $t=$ degree of $F_j^{(w,n)}$. 
	So we get $$h_j^{(w,n)}\sim 1^{\otimes n-m+1}\otimes P_j^{(w,n)}$$ on the subspace generated by basis elements having $e_0$ at the $(n-m+1+\sum_{i=m}^{n-1}\ell(w_i)+\sigma (s))$th place for $s<j$. Using this and  Lemma $\ref{4.2a}$, the claim of part $(1)$ follows. 
	By Lemma $\ref{4.2a}$, it follows that  
	$$(h_j^{(w,n)})^z\sim 1^{\otimes n-m+1\sum_{i=1}^{n-1}\ell(w_i)}\otimes 1^{\otimes\sigma(j)-1}\otimes (\sqrt{1-q^{2CN}}S^*)^z\otimes 1^{\ell(w_n)-\sigma(j)}$$ on the subspace generated by basis element having $e_0$ at the $(n-m+1+\sum_{i=m}^{n-1}\ell(w_i)+\sigma (s))$th place for $s<j$, where $z\in\mathbb{N}$.
	The claim now follows from the observations made just before the lemma.
\end{proof}

\begin{lmma}\label{5.6}
	Let $w\in W_n^{R_m}$ and $\eta_{w}^n$ be the action of  $\mathcal{O}(SO_q(2n+1)/K_q^{R_m,P(R_m^c)})$ on $c_{00}(\mathbb{Z})^{n-m+1}\otimes c_{00}(\mathbb{N})^{\ell(w)} $.  For each $m\leq i<n$, 
	there exist endomorphisms $T_1^i,T_2^i,\cdots ,T_{2i+1}^i$ and $Q_1^i,Q_2^i,\cdots ,Q_{2i+1}^i$ in $\mathcal{O}(SO_q(2n+1)/K_q^{R_m,P(R_m^c)})$ such that for all $1\leq j \leq 2i+1$, we have 
	\begin{align*}
		&T_j^i(u\otimes e_0\otimes e_0\otimes\cdots\otimes e_0)=C\,(\eta_{w_mw_{m+1}\cdots w_{i}}^i(v_j^{2i+1})u)\otimes e_0\otimes \cdots \otimes e_0,\\
		&Q_j^i(u\otimes e_0\otimes e_0\otimes\cdots\otimes e_0)=C\,(\eta_{w_mw_{m+1}\cdots w_{i}}^i(v_j^{2i+1})^*u)\otimes e_0\otimes \cdots \otimes e_0,
	\end{align*}
	where $u\in c_{00}(\mathbb{Z})^{\otimes n-m+1}\otimes c_{00}(\mathbb{N})^{\otimes \sum_{s=m}^i{\ell(w_s)}}$  and $C$ is a nonzero constant.
\end{lmma}
\begin{proof}
	From Proposition \ref{3.3}, we have 
	$$D_{w_1w_2\cdots w_i} \hookrightarrow D_{w_1w_2\cdots w_n}.$$  Let $\lambda_{i}^n$ be the associated embedding map. 	For $1\leq j\leq 2i+1$, define
	\begin{IEEEeqnarray*}{rCl}
		T_j^i&=&\eta_{e}^n(v_{M_n^i}^{M_n^i}) \otimes \pi_w^n(v^{M_n^i}_{\lambda_{i}^n(j)}),\\
		Q_j^i&=&\eta_{e}^n(v_{M_n^i}^{M_n^i})^* \otimes \pi_w^n(v^{M_n^i}_{\lambda_{i}^n(j)})^*.
	\end{IEEEeqnarray*}
	Note that,
	\[
	\eta_{e}^n(v_{M_n^i}^{M_n^i})=\eta_{e}^i(v_{2i+1}^{2i+1}).
	\]
	Take $u\in c_{00}(\mathbb{Z})^{\otimes n-m+1}\otimes c_{00}(\mathbb{N})^{\otimes \sum_{s=m}^i{\ell(w_s)}}$ of the form  $u=u_0\otimes u_{<i} \otimes u_i$, where $u_0 \in c_{00}(\mathbb{Z})^{\otimes n-m+1}, u_{<i} \in  c_{00}(\mathbb{N})^{\otimes \sum_{s=m}^{i-1}{\ell(w_s)}}$ and $u_i \in c_{00}(\mathbb{N})^{\otimes \ell(w_i)}$. By applying Lemma \ref{lowerrank},  equation (\ref{equation}), and the diagram representation of $\pi_w^n$, we have 
	\begin{IEEEeqnarray*}{rCl}
		T_j^i(u\otimes e_0\otimes e_0\otimes\cdots\otimes e_0)&=& \eta_{e}^n(v_{M_n^i}^{M_n^i}) \otimes \pi_w^n(v^{M_n^i}_{\lambda_{i}^n(j)})(u\otimes e_0\otimes e_0\otimes\cdots\otimes e_0)\\
		&=& \eta_{e}^n(v_{M_n^i}^{M_n^i}) \otimes \pi_w^n(v^{M_n^i}_{\lambda_{i}^n(j)})(u_0\otimes u_{<i}\otimes u_i\otimes  e_0\otimes e_0\otimes\cdots\otimes e_0)\\
		&=& \eta_{e}^n(v_{M_n^i}^{M_n^i})(u_0) \otimes \pi_w^n(v^{M_n^i}_{\lambda_{i}^n(j)})( u_{<i} \otimes u_i\otimes e_0\otimes e_0\otimes\cdots\otimes e_0)\\
		&=& C\eta_{e}^i(v_{2i+1}^{2i+1})(u_0) \otimes u_{<i} \otimes  \pi_{w_i}^i(v^{2i+1}_{\lambda_{i}^n(j)})(u_i) \otimes e_0\otimes e_0\otimes\cdots\otimes e_0)\\
		&=&C(\eta_{w_mw_{m+1}\cdots w_{i}}^i(v_j^{2i+1})u)\otimes e_0\otimes \cdots \otimes e_0.\\
	\end{IEEEeqnarray*}
	Proceeding in a similar way, one gets the other part of the claim.  
\end{proof}

\blmma \label{4.6}
Let $w \in W_n^{R_m}$.
Then for each $m \leq i \leq n$, there exist a permutation $\sigma$ of $\{1,2,\cdots,\ell(w_i)\}$ and polynomials $h_{0}^{(w,i)},h_{1}^{(w,i)},\cdots,h_{\ell(w_n)}^{(w,i)} $ and $h_{1*}^{(w,i)},h_{2*}^{(w,i)},\cdots,h_{\ell(w_n)*}^{(w,i)}$ with variables $\eta_w^n(v^k_l)$
with $v^k_l\in\zeta_m$ such that
\begin{enumerate}[(i)]
	\item 
	\begin{IEEEeqnarray*}{lCl}
		\overleftarrow{\prod_{i=m}^{n}}(h_{1*}^{(w,i)})^{p^i_{\sigma_{i}(1)}}
		(h_{1}^{(w,i)})^{r^i_{\sigma_{i}(1)}}\cdots(h_{\ell(w_i)^*}^{(w,i)})^{p^i_{\sigma_{i}
				(\ell(w_i))}}(h_{\ell(w_i)}^{(w,i)})^{r^i_{\sigma_{i}(\ell(w_i))}}
		\overleftarrow{\prod_{i=m}^n}(h_{0}^{(w,i)})^{r^i_{0}}(e_0\otimes\cdots\otimes e_0)\\
		=C\,e_{r_0^n}\otimes e_{r_0^{n-1}}\otimes\cdots\otimes e_{r_0^k}\otimes e_{r_{r_1^k-p_1^k}}\otimes   \cdots \otimes  e_{{r_{\ell(w_k)}^k-p_{\ell(w_k)}^k}}\otimes\cdots\otimes  e_{{r_1^n-p_1^n}}\otimes \cdots\otimes 
		e_{{r_{\ell(w_n)}^n-p_{\ell(w_n)}^n}},
	\end{IEEEeqnarray*}
	where $r_0^j,r_j^i,p_j^i\in\mathbb{N}$ and $r_j^i \geq p_j^i$ for $1\leq j\leq \ell(w_i)$ and $m\leq i\leq n.$
	\item The following endomorphisms
	$$\overleftarrow{\prod_{i=m}^{n}}(h_{1*}^{(w,i)})^{p^i_{\sigma_{i}(1)}}
	(h_{1}^{(w,i)})^{r^i_{\sigma_{i}(1)}}\cdots
	(h_{\ell(w_i^*)}^{(w,i)})^{p^i_{\sigma_{i}(\ell(w_i))}}\\(h_{\ell(w_i)}^{(w,i)})^{r^i_{\sigma_{i}
			(\ell(w_i))}}\overleftarrow{\prod_{i=m}^n}(h_{0}^{(w,i)})^{r^i_{0}},$$ 
	
	where $r_0^j,r_j^i,p_j^i\in\mathbb{N}$, $r_j^i \geq p_j^i$ for $1\leq j\leq \ell(w_i)$, and $m\leq i\leq n$ are linearly independent.
\end{enumerate}
\elmma
\begin{proof}
	Let $h_j^{(w,n)}$, $h_{j\ast}^{(w,n)}$,  $1\leq j \leq \ell(w_n)$, and $\sigma_n$ be as given in Lemma \ref{5.5}.  To define polynomials $h_j^{(w,i)}$ and $h_{j\ast}^{(w,i)}$ for $0\leq j\leq \ell(w_i)$ and permutation $\sigma_i$  and $k\leq i<n$, we first view $w_1w_2\cdots w_i$ as an element of the Weyl group $W_i$ of $\mathfrak{so}(2i+1)$.  By Lemma $\ref{5.5}$,  one can define polynomial $h_j^{(w_1w_2\cdots w_i,i)} $ and $h_{j\ast}^{(w_1w_2\cdots w_i,i)}$ and permutation $\sigma_i$.  Now, we   replace the variable $\eta_{w_1w_2\cdots w_i}^i(v_l^{2i+1})$ with $T_l^i$ and $\eta_{w_1w_2\cdots w_i}^i(v_l^{2i+1})^*$ with $Q_l^i$ for $1\leq l \leq 2i+1$ in the polynomial $h_j^{(w_1w_2\cdots w_i,i)} $ and
	$h_{j\ast}^{(w_1w_2\cdots w_i,i)}$, and  obtain the  polynomials  $h_j^{(w,n)}$ and $h_{j\ast}^{(w,n)}$, respectively, for all $0\leq j\leq \ell(w_i)$.  Applying  Lemma $\ref{5.5}$ and Lemma $\ref{5.6}$, we get both parts of the claim. 
\end{proof}
\begin{thm}\label{4.7}
	Let $w_n^{R_m}$ be the longest element of  $W^{R_m}_n$. Then we have the following:
	$$\text{GKdim}(\mathcal{O}(\mbox{Spin}_q(2n+1)/K_q^{R_m,P(R^c_m)}))= 2\ell(w_n^{R_m})+n-m+1=\text{dim}(\mbox{Spin}(2n+1)/K^{R_m,P(R^c_m)}).$$
\end{thm}
\begin{proof}
	To prove the first equality, it suffices to show that
	$$\text{GKdim}(\mathcal{O}(\mbox{Spin}_q(2n+1)/K_q^{R_m,P(R^c_m)}))\geq 2\ell(w_n^{R_m})+n-m+1,$$
	thanks to Proposition \ref{4.5}.
	We will work with the image of $\mathcal{O}(\mbox{Spin}_q(2n+1)/K_q^{R_m,P(R^c_m)})$ under the faithful  representation  $\eta_{w_n^{R_m}}^n$.
	Take a generating set $F=\eta_{w_n^{R_m}}^n(\zeta_m)\cup\eta_{w_n^{R_m}}^n({(\zeta_m})^*)$.
	Define $$A:=\max\{ \text{degree of }h_j^{(w_n^{R_m},i)}:0\leq j\leq \ell({(w_n^{R_m})}_i),m\leq i\leq n\}.$$ From part $(2)$ of Lemma $\ref{4.6}$, we get
	\begin{align*}
		GKdim(\mathcal{O}(\mbox{Spin}_q(2n+1)/K_q^{R_m,P(R^c_m)}))\geq & \limsup \frac{\text{ln} ~dim(F^{Ar})}{\text{ln} Ar}\\\geq &\limsup \frac{\text{ln} \frac{{r+2\ell(w_n^{R_m})+n-m \choose r}}{2}}{\text{ln} Ar}=2\ell(w_n^{R_m})+n-m+1.
	\end{align*}
	Now,  consider the complexified Lie algebra ${\mathfrak{k}^\mathbb{C}_{R}}$  of $K^{R_m,P(R^c_m)}$. Let $w_{n,R_m}$ be the longest word in the Weyl group $W_{n,R_m}$ of ${\mathfrak{k}^\mathbb{C}_{R}}$. Using the decomposition of ${\mathfrak{k}^\mathbb{C}_{R}}$  into roots spaces and Cartan subalgebra (see page $226$ of \cite{NesTus-2012aa}), and  that  twice the length of the longest word of a Weyl groups is the same as the number of roots
	and the rank of a maximal
	torus is the same as the dimension of a cartan subalgebra of  ${\mathfrak{k}^\mathbb{C}_{R}}$, we have
	$$\text{dim}~(K^{R_m,P(R^c_m)})=\text{dim}({\mathfrak{k}^\mathbb{C}_{R}})=2\ell(w_{n,R_m})+m-1.$$
	Let $w_{n,\Pi}$ be the longest element of $W_n$ . By applying the same   argument as above, we get
	$$dim (\mbox{Spin}(2n+1))=dim (\mathfrak{so}(2n+1))=2\ell(w_{n,\Pi})+n.$$
	Since $w_{n,\Pi}= w_{n,R_{m}} w_n^{R_m}$ with 
	$\ell(w_{n,\Pi})=\ell(w_{n,R_m})+\ell(w_n^{R_m})$, it follows that 
	\begin{align*}
		\text{dim}~(\mbox{Spin}(2n+1)/K^{R_m,P(R^c_m)})=&\text{dim} (\mbox{Spin}(2n+1))-\text{dim}(K^{R_m,P(R^c_m)})\\
		=&(2\ell(w_{n,\Pi})+n)-(2\ell(w_{n,{R_m}})+m-1)=2\ell(w_n^{R_m})+n-m+1.
	\end{align*}
\end{proof}
\begin{crlre} One has 
	$$GKdim(\mathcal{O}(\mbox{Spin}_q(2n+1)/\mbox{Spin}_q(2n-2m+1)))
	=dim(\mbox{Spin}(2n+1)/\mbox{Spin}(2n-2m+1)).$$
\end{crlre}
\begin{proof}
	The claim follows immediately from  Theorem $\ref{4.7}$ and the fact that  $$\mathcal{O}((\mbox{Spin}_q(2n+1)/\\K_q^{R_{n-m+1},P(R^c_{n-m+1})})
	=\mathcal{O}(\mbox{Spin}_q(2n+1)/\mbox{Spin}_q(2n-2m+1)).$$
\end{proof}
\begin{thm}\label{4.8}	Let $w_n^{R_m}$ be the longest element of  $W^{R_m}_n$. Then we have the following:
	$$\text{GKdim}(\mathcal{O}((SO_q(2n+1)/\tilde{K}_q^{R_m,P(R^c_m)}))= 2\ell(w_n^{R_m})+n-m+1=\text{dim}(SO(2n+1)/\tilde{K}^{R_m,P(R^c_m)}).$$
\end{thm}
\prf Observe  that   Proposition \ref{4.5} holds in this case as well.  The same proof works. 
This gives us the following inequality.
$$\text{GKdim}(\mathcal{O}(SO_q(2n+1)/\tilde{K}_q^{R_m,P(R^c_m)}))\geq  2\ell(w_n^{R_m})+n-m+1.$$
Next, note that the the polynomials mentioned in Lemma \ref{4.6} involves variables of matrix entries of first and last  $m$ rows $(\!(v_{l}^{k})\!)$ which are in the quotient algebra $\mathcal{O}(SO_q(2n+1)/\tilde{K}_q^{R_m,P(R^c_m)})$. Proceeding in the same way as in Theorem \ref{4.7}, we get the equality. 
\[
\text{GKdim}(\mathcal{O}(SO_q(2n+1)/\tilde{K}_q^{R_m,P(R^c_m)}))\leq  2\ell(w_n^{R_m})+n-m+1.
\]
This proves our first result. The rest immediately follows from the facts that  
\[
\mbox{dim}(\mbox{Spin}(2n+1))= \mbox{dim}(SO(2n+1)), 
\]
and  $$\mbox{dim}(K_q^{R_{n-m+1},P(R^c_{n-m+1})})
=\mbox{dim}(\tilde{K}_q^{R_{n-m+1},P(R^c_{n-m+1})}).$$
\qed
\begin{crlre} 
	$$GKdim(\mathcal{O}(SO_q(2n+1)/SO_q(2n-2m+1)))=dim(SO(2n+1)/SO(2n-2m+1)).$$
\end{crlre}
\prf  Note that  $$\mathcal{O}(SO_q(2n+1)/\tilde{K}_q^{R_{n-m+1},P(R^c_{n-m+1})})=
\mathcal{O}(SO_q(2n+1)/SO_q(2n-2m+1)).$$
The claim now follows from Theorem \ref{4.8}.
\qed
\brmrk  A few remarks are in order.
\begin{enumerate}
	\item If one looks at the proof carefully, it relies mainly upon two pillars: firstly, explicit construction of certain polynomials of generators for the last part of the Weyl word, and secondly, the diagram embedding properties, which allows us to use step-wise induction. In (\cite{ChaSau-2018aa}, \cite{ChaSau-2019aa}), and in this article, the first step is somewhat natural. However, the second step is done on a case-by-case basis. It would be very interesting if we find a natural way of proving the diagram embedding properties. This would certainly extend the results of this paper to the $q$-deformation of a more general class of semisimple Lie groups.
	\item Though in this paper, we have chosen  a subset $R_m$ of the set of simple roots $\Pi$ of a particular form, one can take any subset $R\subset \Pi$ and prove the same results.  The same proof works.  
	\item Here we need  diagram embeddability  only for the irreducible representations of $\mathcal{O}(SO_q(2n+1))$ and its quotients space $\mathcal{O}((SO_q(2n+1)/\tilde{K}_q^{R_{m},P(R^c_{m})})$ to get our results.  However, this property holds for  $\mathcal{O}(\mbox{Spin}_q(2n+1))$ and its quotients space $\mathcal{O}((\mbox{Spin}_q(2n+1)/K_q^{R_{m},P(R^c_{m})})$.
\end{enumerate}

\ermrk
\noindent\begin{footnotesize}\textbf{Acknowledgement}:
	Bipul Saurabh  acknowledges support from  NBHM grant 02011/19/2021/NBHM(R.P)/R\&D II/8758.
\end{footnotesize}

\vspace{1cm}
\noindent{\sc Akshay Bhuva} (\texttt{akshayb@iitgn.ac.in, akshaybhuva98@gmail.com})\\
         {\footnotesize Indian Institute Of Technology, Gandhinagar, Palaj, Gandhinagar, 382355, INDIA}\\

\noindent{\sc Bipul Saurabh} (\texttt{bipul.saurabh@iitgn.ac.in})\\
         {\footnotesize Indian Institute Of Technology, Gandhinagar, Palaj, Gandhinagar, 382355, INDIA}

\end{document}